\documentclass[a4,11pt]{article}%
\usepackage{amsmath}
\usepackage{amsfonts}
\usepackage{amssymb}
\usepackage{graphicx}%
\usepackage{color}
\usepackage{amsthm}
\usepackage{mathrsfs}
\providecommand{\U}[1]{\protect\rule{.1in}{.1in}}

\providecommand{\U}[1]{\protect\rule{.1in}{.1in}}
\providecommand{\U}[1]{\protect\rule{.1in}{.1in}}
\RequirePackage[colorlinks,citecolor=blue,urlcolor=blue]{hyperref}
\RequirePackage{hypernat}
\providecommand{\U}[1]{\protect\rule{.1in}{.1in}}
\numberwithin{equation}{section}
\newtheorem{lemma}{Lemma}[section]
\newtheorem{theorem}{Theorem}[section]
\newtheorem{corollary}{Corollary}[section]
\newtheorem{proposition}{Proposition}[section]
\newtheorem{definition}{Definition}[section]

\newtheorem{remark}{Remark}[section]

\allowdisplaybreaks

\newcommand{\abs}[1]{|#1|}

\def \R {{\bf R}}
\def \d {\delta}

\def \esp {{[0,T]\times \R^m}}
\def \P {{\cal P}}
\def \m {{\cal M}}
\def \cf {\mathcal{F}}
\def \bp {{\bf{P}}}
\def \be {{\bf E}}
\def \bP {{\bf P}}
\def \bs {\bigskip}
\def \ms {\medskip}
\def \no {\noindent}
\def \tx {(t,x)\in \esp}
\def \cM {\mathcal{M}}

\def \cP {\mathcal{P}}
\def \rw {\rightarrow}

\def \qq {\qquad}
\def \dt {\dots}

\begin{document}

\markboth{Said Hamad\`ene and Rui Mu}{Stochastics: An International Journal of Probability and Stochastic Processes}

%

\title{\vspace{-1in}\parbox{\linewidth}{\footnotesize\noindent
} \vspace{\bigskipamount} \\Existence of Nash Equilibrium Points for Markovian Nonzero-sum Stochastic Differential Games with Unbounded Coefficients }
\author{Said Hamad\`ene$^{1}$,\ \ Rui Mu$^{1 2}$\\$^{1}$Universit\'e du Maine, LMM, Avenue Olivier Messiaen, 72085 Le Mans, Cedex 9, France\\$^{2}$School of Mathematics, Shandong University, Jinan 250100, China\\E-mails: hamadene@univ-lemans.fr; rui.mu.sdu@gmail.com}
\date{\today}
\maketitle

\maketitle

\begin{abstract}
This paper is related to nonzero-sum stochastic differential games in the Markovian framework. We show existence of a Nash equilibrium point for the game when the drift is no longer bounded and  only satisfies a linear growth condition. The main tool is the notion of backward stochastic differential equations which, in our case, are multidimensional with continuous coefficient and stochastic linear growth.  
\end{abstract}
\medskip

\bs


\section{Introduction}
In this work, we analyze a nonzero-sum stochastic differential game (NZSDG for short) which is described as follows. Assume one has $N$ players $\pi_1, ..., \pi_N$ which intervene on (or control) a system. Each one with the help of an admissible control which is an adapted stochastic process $u^i:=(u_t^i)_{t\leq T}$ for $\pi_i$, $i=1,\dt,N$. When the $N$ players make use of a strategy $u:=(u^1,...,u^N)$, the dynamics of the controlled system is a process $(x^u_t)_{t\leq T}$  solution of the following standard stochastic differential equation (SDE for short):
\begin{equation}\label{sdeg}
dx^u_t=f(t,x^u_t,u_t^1,...,u^N_t)dt+\sigma(t,x^u_t)dB_t \mbox{ for }t\leq T \, \mbox{ and } \,x_0=x\,;\end{equation}
$B:=(B_t)_{t\leq T}$ is a Brownian motion. The control actions are not free and generate for each player $\pi_i$, $i=1,...,N$, a payoff which amounts to
$$
J_i(u^1,...,u^N):=\be[\begin{array}{c}g^i(x^u_T)+\int_0^Th_i(s,x^u_s,u_s)ds\end{array}].$$
A Nash equilibrium point (NEP for short) for the players is a strategy $u^*:=(u^{1,*}, ...,u^{N,*})$ of control of the system which has the feature that each player $\pi_i$ who takes unilaterally the decision to deviate from $u^{i,*}$, is penalized: For all $i=1,...,N$, for all control $u^i$ of player $\pi_i$, 
$$
J_i(u^*)\leq J_i([u^{*,-i}|u^i])$$
where $[u^{*,-i}|u^i]:=(u^{1,*},\dt, u^{i-1,*},u^i,u^{i+1,*},\dt, u^{N,*}).$

In the case when $N=2$ and $J_1+J_2=0$, this game reduces to the well known zero-sum differential game which is well documented in several works and from several points of view (see e.g. \cite{bensoussan-lions}, \cite{buckdahn1}, \cite{daviselliott}, \cite{elliott}, \cite{flemingsouganidis},  \cite{friedman}, \cite{hamadene1995zero},  \cite{hamadene1995backward}, \cite{isaacs} etc. and the references inside).

Comparatively, the nonzero-sum differential game is so far less considered even though there are some works on the subject, including  \cite{buckdahn2004},  \cite{friedman2}, \cite{hamadene1997bsdes}, 
\cite{hamadene1},  \cite{hamadene3}, \cite{hamadene4}, \cite{lepeltier}, \cite{qianlin}, \cite{mannucci}, \cite{rainer}, etc.). In these works, the objectives are various and so are the approaches,  usually based on partial differential equations (PDEs) (\cite{friedman2, mannucci}) or backward SDEs (\cite{hamadene1997bsdes, hamadene3, hamadene4, qianlin, lepeltier}). On the other hand, it should be pointed out that the frameworks in those papers are not the same. Some of them consider strategies as control actions for the players (e.g. \cite{buckdahn2004}, \cite{qianlin},  \cite{rainer}) while others deal with the control against control setting (e.g. \cite{hamadene1}, \cite{hamadene1997bsdes, hamadene3, hamadene4}). 
The first ones, formulated usually in the framework of two players, allow to study the case where the diffusion coefficient $\sigma$ is controlled. In the latter ones, $\sigma$ does not depend on the controls. However those papers do not reach the same objective. Note that for the control against control zero-sum game, Pham and Zhang \cite{pham} and M.Sirbu \cite{sirbu} have overcome this restriction related to the independence of $\sigma$ on the controls.  

In the present article, we study the nonzero-sum game of type control against control with the diffusion process  $\sigma$ independent of controls, in the same line as in the paper by Hamad\`ene et al. \cite{hamadene1997bsdes}. But in \cite{hamadene1997bsdes}, the setting concerns only the case when the coefficients $f$ and $\sigma$ of the diffusion (\ref{sdeg}) are bounded. According to our knowledge the setting where those coefficients are not bounded and of linear growth is not considered yet. Therefore the main objective of this work is to relax as much as possible the boundedness of the coefficients $f$ (mainly) and $\sigma$(which is not bounded as stated in the final extension). The novelty of the paper is that we show the existence of a Nash equilibrium point for the NZSDG when $f$ is no longer bounded but only satisfies the linear growth condition. As in \cite{hamadene1997bsdes} our approach is based on backward SDEs and basically the problem turns into studying its associated multi-dimensional BSDE which is of linear growth $\omega$ by $\omega$. Under the generalized Isaacs hypothesis and the domination condition of laws of solutions of (\ref{sdeg}), which is satisfied when H\"ormander's condition on $\sigma$ is satisfied, we show that the latter BSDE has a solution which then provides a NEP for the NZSDG. 

The paper is organized as follows:

In Section 2 we fix the setting of the problem and recall some results which play an important role in our study. The formulation we adopt is of weak type. On the other hand for the sake of simplicity we have made the presentation for $N=2$. The generalization to the situation where $N\geq 3$ is formal and can be carried out in the same spirit. Section 3 is devoted to the link between the game and BSDEs. We first express the payoffs of the game in using solutions of BSDEs whose integrability is not standard. Then we show that the existence of a NEP for the game turns into the existence of a solution of a specific BSDE which is of multi-dimensional type and linear growth $\omega$ by $\omega$.  It plays a role of a verification theorem for the NZSDG.  In Section 4 we show that this specific BSDE has a solution when the generalized Isaacs condition is fulfilled and the laws of the dynamics of the non-controlled system satisfy the so-called $L^q$-domination condition. This latter is especially satisfied when the diffusion coefficient $\sigma$ satisfies the well known H\"ormander's condition. 
Our method is based on: (i) the introduction of an approximating scheme of BSDEs which is well-posed since the coefficients are Lipschitz. In this markovian framework  of randomness, the solutions $(Y^n,Z^n)$, $n\geq 1$, of this  scheme can be represented via deterministic functions $(\varpi^n,\upsilon^n)$, $n\geq 1$, and the Markov process as well ; (ii) sharp estimates for $(Y^n,Z^n)$ and $(\varpi^n,\upsilon^n)$ and the $L^q$-domination condition enable us to obtain the strong convergence of a subsequence $(\varpi_{n_k})_{k\geq 1}$ from a  weak convergence in an appropriate space. This yields the strong convergence of the corresponding subsequences $(Y^{n_k})_{k\geq 1}$ and $(Z^{n_k})_{k\geq 1}$ ; (iii) we finally show that the limit of 
$(Y^{n_k}, Z^{n_k})_{k\geq 1}$ is a solution for the BSDE associated with the NZSDG. At the end of this section we provide an example which illustrates our result. We also discuss possible extensions of our findings to the case when both the drift $f$ and diffusion coefficient $\sigma$ of (\ref{sdeg}) are not bounded. 
$\Box$
\section{Setting of the problem}\label{section of setting of the
problem}

Let $T>0$ and let $(\Omega, \mathcal{F}, \bP)$ be a probability space on which is
defined an $m$-dimensional Brownian motion $B:=(B_t)_{0\leq t\leq T}$.
For $t\leq T$, let us denote by $(F_t:=\sigma(B_u, u\leq t))_{t\leq T}$ the natural filtration of $B$ and $(\mathcal{F}_t)_{t\leq T}$ the completion of $(F_t)_{t\leq T}$ with the $\bp$-null sets of $\cf$, which then satisfies the usual conditions. Let $\mathcal{P}$ be the
$\sigma$-algebra on $[0,T]\times \Omega$ of
$\mathcal{F}_t$-progressively measurable sets. Let $p\ge 1$ be a real constant and $t\in [0,T]$ fixed. We then define the following spaces:

\begin{itemize}
\item   $L^p_T(\R^m)=\{\xi: \mathcal{F}_T$-measurable and $\R^m$ -valued random variable s.t. $\be[|\xi|^p]<\infty\}$;
 \item  $\mathcal{S}_{t,T}^p(\R^m)=\{\varphi=(\varphi_s)_{t\leq s\leq T}: \mathcal{P}$ -measurable and $\R^m$ -valued s.t. $\be[sup_{t\leq s\leq T}|\varphi_s|^p]< \infty \}$ ; $ \mathcal{S}_{0,T}^p(\R^m)$ is simply denoted by $
\mathcal{S}_{T}^p(\R^m)$;
 \item $\mathcal{H}_{t,T}^p(\R^m)=\{Z=(Z_s)_{t\leq s\leq T}: \mathcal{P}$ -measurable and $\R^m$ -valued s.t. $\be[(\int_t^T|Z_s|^2ds)^{p/2}]< \infty
 \}$ ; $\mathcal{H}_{0,T}^p(\R^m)$ is simply denoted by 
$\mathcal{H}_{T}^p(\R^m)$. 
\end{itemize}
\medskip

\noindent Next let $\sigma$ be a matrix function defined as:
$$\begin{array}{cl}\sigma: [0,T]\times \R^m&\longrightarrow \R^{m\times
m}\\ (t,x)&\longmapsto \sigma (t,x)\end{array}$$ and which satisfies the following assumptions:
\begin{flushleft}
    \textbf{Assumptions (A1)}
\end{flushleft}
\begin{enumerate}
  \item[(i)] $\sigma$ is uniformly Lipschitz w.r.t $x$. \textit{i.e.} there exists a constant $C_1$  such that,
  $$\forall t \in [0, T], \forall \ x, x^{\prime} \in \R^m,\quad
\abs{ \sigma(t,x)-\sigma(t, x^{\prime})} \leq C_1 \abs{x-
x^{\prime}}.$$
  \item[(ii)] $\sigma$ is invertible and bounded and its inverse is bounded, \textit{i.e.}, there exits a constant $C_2$ such that $$\forall (t,x)\in
  \esp, \quad \abs {\sigma(t, x)} +\abs{\sigma^{-1}(t,x)}\leq C_2.$$
  \end{enumerate}
\medskip

\begin{remark} H\"ormander's condition.

\noindent Under (A1), there exists a real constant $\Upsilon>0$ such that for any $(t,x)\in \esp $,
\begin{equation}\label{horm}
\Upsilon.I\leq \sigma(t,x).\sigma^\top(t,x)\leq \Upsilon^{-1}.I
\end{equation}
where $I$ is the identity matrix of dimension $m$. $\Box$
\end{remark}

Next let $(t,x)\in [0,T]\times \R^m$. Under the Assumptions (A1), (i)-(ii), we know that there exists a process
$(X^{t,x}_s)_{s\leq T}$ that satisfies the following stochastic
differential equation (see e.g. Karatzas and Shreve, pp.289, 1991
\cite{karatzas1991brownian}):
\begin{equation*}
X^{t,x}_s= x+ \int_t^s \sigma(r,X^{t,x}_r)dB_r, \quad \forall s\in [t,T] \mbox{ and }
X^{t,x}_s=x \mbox{ for }s\in [0,t].
\end{equation*}

Hereafter for sake of simplicity we will deal with the setting of two players. However the generalization to the case of $N (\geq 3)$ players is formal and just a question of writing (see the comment of Remark \ref{generalisation}). Also let us denote by $U_1$ and $U_2$ two compact metric spaces and
let $\m _1$ (resp. $\m _2$) be the set of $\P$-measurable processes
$u=(u_t)_{t\leq T}$ (resp. $v=(v_t)_{t\leq T}$) with values in $U_1$
(resp. $U_2$). We denote by $\m$ the set $\m_1\times \m_2$ and call it the set of admissible controls for the players.
\medskip

Let $f$ be a Borelian function from $[0,T]\times \R^m \times U_1
\times U_2$ into $\R^m$ and for $i= 1, 2$ let $h_i$ and $g^i$ be
Borelian functions from $[0,T]\times \R^m \times U_1 \times U_2$
(resp. $\R^m$) into $\R$ which satisfy:
\begin{flushleft}
    \textbf{Assumptions (A2)}
\end{flushleft}
\begin{enumerate}
  \item[(i) ] $f$ is of linear growth w.r.t $x$, \textit{i.e.} there exists a constant $C_3$
  such that $\abs{f(t,x,u,v)}\leq C_3(1+\abs{x})$,  for any $(t,x,u,v)\in [0,T]\times \R^m\times U_1\times U_2.$
  \item[(ii)] for $i=1,2$ $h_i$ is of polynomial growth w.r.t $x$, \textit{i.e.}, there exists a constant $C_4$ and $\gamma \geq 0$ such that $\abs{h_i(t,x,u,v)}\leq C_4(1+\abs{x}^{\gamma})$ for any $(t,x,u,v)\in [0,T]\times \R^m\times U_1\times U_2.$
  \item[(iii)] for $i=1,2$, $g^i$ is of polynomial growth with respect to $x$, \textit{i.e.} there exist constants $C_5$ and $\gamma \geq 0$ such that $\abs{g^i(x)}\leq C_5(1+\abs{x}^{\gamma}),
\forall x\in \R^m$.$\Box$
\end{enumerate}

For $(u,v)\in \m$, let $\bp^{(u,v)}_{(t,x)}$ be the measure on $(\Omega,
\mathcal{F})$ whose density function is defined as follows:
\begin{equation}\label{defzeta}
\frac{d\bp^{(u,v)}_{(t,x)}}{d\bP} = \zeta_T\left(\sigma^{-1}(\cdot, X^{t,x}_{\cdot})f(\cdot,
X_{\cdot}^{t,x}, u_{\cdot}, v_{\cdot})\right),
\end{equation}
\noindent where for any measurable $\mathcal{F}_t$-adapted process $\eta:=(\eta_s)_{s\leq T}$ we define,  \begin{equation}\label{density function}\zeta_s(\eta) := e^{\int_0^s \eta_rdB_r- \frac{1}{2}\int_0^s
\left|\eta_r\right|^2dr},\,\forall s\leq T.
\end{equation}

\no Thanks to the Assumptions (A1) and (A2)-(i) on $\sigma$ and
$f$, we can infer that $\bp^{(u,v)}_{(t,x)}$ is a probability on $(\Omega,
\mathcal{F})$ (see Appendix A of N. El-Karoui and S. Hamad\`ene \cite{hamadene2003} or Karatzas-Shreve \cite{karatzas1991brownian}, pp.200). Then by Girsanov's theorem (Girsanov, \cite{girsanov1960transforming}), the process $B^{(u,v)}:=
(B_s-\int_0^s \sigma^{-1}(r,X^{t,x}_r)f(r, X^{t,x}_r, u_r, v_r) dr)_{s\leq T}$
is a $(\mathcal{F}_s, \bp^{(u,v)}_{(t,x)})$-Brownian motion and
$(X^{t,x}_s)_{s\leq T}$ satisfies the following stochastic differential
equation,
\begin{equation}\label{new equation of x}
dX^{t,x}_s\!=\!f(s, X^{t,x}_s, u_s, v_s)ds+ \sigma(s, X^{t,x}_s)dB_s^{(u,v)},\,\,\mbox{for }s\in [t,T] \mbox{ and }X^{t,x}_{s}= x
\mbox{ for }s<t.
\end{equation}
\noindent In general, the process $(X^{t,x}_s)_{s\leq T}$ is not
adapted with respect to the filtration generated by the Brownian motion
$(B_s^{(u,v)})_{s\leq T}$, therefore $(X^{t,x}_s)_{s\leq T}$ is called
a weak solution for the SDE (\ref{new equation of x}).
\medskip

Next let us fix $(t,x)$ to $(0,x_0)$ and for $i=1,2$, let us define the payoffs of the players by:
\begin{equation}\label{cost function}
J^i(u,v)= \be^{(u,v)}_{(0,x_0)}[\int_0^T h_i(s, X^{0,x_0}_s, u_s, v_s)ds+
g^i(X^{0,x_0}_T)],
\end{equation}
\noindent where $\be^{(u, v)}_{(0,x_0)}(\cdot)$ is the expectation under the
probability $\bp^{(u,v)}_{(0,x_0)}$. Hereafter $\be^{(u, v)}_{(0,x_0)}(\cdot)$ (resp. $\bp^{(u,v)}_{(0,x_0)}$) will be simply denoted by 
$\be^{(u, v)}(.)$ (resp. $\bp^{(u,v)}$). \\

Our problem is to find an admissible control $(u^*, v^*)$ such that
\begin{equation*}
J^1(u^*, v^*)\leq J^1(u, v^*) \mbox{ and }
J^2(u^*, v^*)\leq J^2(u^*,
v) \,\,\mbox{ for any }(u,v)\in {\cal M}.
\end{equation*}
The control $(u^*, v^*)$ is called a \textit{Nash
equilibrium point} for the nonzero-sum stochastic differential game.
\medskip

Next we define the Hamiltonian functions $H_i$, $i=1,2$, of the game from $[0, T]\times \R^{2m}\times U_1\times U_2$ into $\R$ by:
\begin{equation*}
H_i(t, x, p, u, v)= p\sigma^{-1}(t,x)f(t, x, u, v)+ h_i(t, x, u, v),
\end{equation*}
and we introduce the following assumption (A3) called
the generalized \textit{Isaacs condition}.
\medskip

\noindent\textbf{Assumption (A3)}: Generalized Isaacs condition.\\
\\
(i) There exist two Borelian applications $u_1^*$, $u_2^*$ defined
  on $[0,T]\times \R^{3m}$, with values in $U_1$ and $U_2$, respectively, 
  such that for any $(t, x, p, q, u, v)\in [0, T]\times \R^{3m}\times
  U_1\times U_2$, we have:
  $$H_1^*(t,x,p,q)= H_1(t, x, p, u_1^*(t,x,p,q), u_2^*(t,x,p,q))\leq
  H_1(t,x,p,u,u_2^*(t,x,p,q))$$
and
  $$H_2^*(t,x,p,q)= H_2(t, x, q, u_1^*(t,x,p,q), u_2^*(t,x,p,q))\leq
  H_2(t,x,q,u_1^*(t,x,p,q), v).$$
  \medskip
\noindent (ii) the mapping $(p,q)\in \R^{2m}\longmapsto
(H_1^*,H_2^*)(t,x,p,q) \in \R$ is continuous for any fixed $(t,x)$.
$\Box$
\medskip

\begin{remark}This condition has been already considered by A. Friedman in \cite{friedman} for the same purpose as ours in this paper. But the treatment of the problem he used is the PDE approach. $\Box$

\end{remark}

In order to show that the game has a Nash equilibrium point, it is
enough to show that its associated BSDE, which is multi-dimensional
and with a continuous generator (see Theorem \ref{bsde we mainly proof} below) has a solution. Therefore the main objective of the next section is to study the connection between NZSDGs and BSDEs.

\section{Nonzero-sum differential game problem and BSDEs.}

Let $(t,x)\in \esp$ and $(\theta_s^{t,x})_{s\leq T}$ be the solution of
the following forward stochastic differential equation:
\begin{equation}\label{forward equation of x}
\left\{
\begin{aligned}
d\theta_s&= b(s, \theta_s)ds+ \sigma(s, \theta_s)dB_s, \qquad &s \in[t, T];\\
\theta_s&= x,  &s \in[0, t]\\
\end{aligned}
\right.
\end{equation}

\noindent where $\sigma:[0, T]\times \R^m\rightarrow \R^{m\times m}$
satisfies the Assumptions (A1) and $b$: $[0, T]\times
\R^m\rightarrow \R^m$ is a measurable function which satisfies the
following assumption:
\bigskip

\noindent \textbf{Assumption (A4)}: The function $b$ is uniformly Lipschitz w.r.t $x$ and of linear growth, \textit{i.e.}, there exist constants $C_6$ and $C_7$  such that:
  $$\forall t \in [0, T],\  \forall \ x, x^{\prime} \in \R^m,\
\abs{ b(t,x)-b(t, x^{\prime})} \leq C_6 \abs{x- x^{\prime}}\mbox{ and } \abs {b(t, x)} \leq C_7(1+ \abs{x}).$$

It is well-known that, under (A1) and (A4), the stochastic process $(\theta_s^{t,x})_{s\leq T}$ satisfies
the following estimate, see for example (Karatzas, I. 1991
\cite{karatzas1991brownian} pp.306):
\begin{equation}\label{estimate of X}
\forall \ q \in [1, \infty ), \qquad \be\Big[\big(\sup\limits_{s\leq
T}\abs{\theta_s^{t,x}}\big)^{2q}\Big]\leq C(1+\abs{x}^{2q}).
\end{equation}
As a particular case we have a similar estimate for the process $X^{t,x}$, i.e., 
\begin{equation}\label{estimate of X1}
\forall \ q \in [1, \infty ), \qquad \be\Big[\big(\sup\limits_{s\leq
T}\abs{X_s^{t,x}}\big)^{2q}\Big]\leq C(1+\abs{x}^{2q}).
\end{equation}

Finally note that we have also a similar estimate for weak solutions of SDEs of types (\ref{new equation of x}), i.e., if $(u,v)$ belongs to $\mathcal{M}$ then 
\begin{equation}\label{estimate of X2}
\forall \ q \in [1, \infty ), \qquad \be^{(u,v)}_{(t,x)}\Big[\big(\sup\limits_{s\leq
T}\abs{X_s^{t,x}}\big)^{2q}\Big]\leq C(1+\abs{x}^{2q}).
\end{equation}

Next let us recall the following result by U.G.Haussmann
\cite{haussmann1986stochastic} related to integrability of the exponential local martingale  defined in (\ref{density function}).
\begin{lemma}\label{density function lp bounded}
Assume (A1)-(i),(ii) and (A4) and let $(\theta^{t,x}_s)_{s\leq T}$ be the solution of
(\ref{forward equation of x}). Let $\varphi$ be a $\mathcal{P}\otimes \mathcal{B}(\R^m)$-measurable application from $[0,T]\times \Omega \times \R^m$ to $\R^m$ which is uniformly of linear growth, that is, $\bp$-a.s., $\forall (s,x)\in \esp$, 
\begin{equation*}\abs{\varphi(s,\omega,x)}\leq
C_8(1+\abs{x}).
\end{equation*}
\noindent Then, there exists $p\in (1,2)$ and a constant $C$, where
$p$ depends only on $C_2$, $C_6$, $C_7$, $C_8$, $m$ while the constant $C$,
depends only on $m$ and $p$, but not on $\varphi$, such that:
\begin{equation*}
\be\left[\big(\zeta_T (\varphi(s, \theta^{t,x}_s))\big)^p\right] \leq
C,
\end{equation*}
\noindent where the process $\zeta(\varphi(s, \theta^{t,x}_s))$ is the density function defined in (\ref{density
function}).
\end{lemma}
As a by-product we have:
\begin{corollary}\label{corol lp} Let $(u,v)$ be an admissible control for the players and $(t,x)\in \esp$. Then there exists $p>1$ such that  
\begin{equation}
\be\left[\big(\zeta_T(\sigma(s,X^{t,x}_s)^{-1}f(s,X^{t,x}_s,u_s,v_s))\big)^p\right] \leq
C. 
\end{equation}
\end{corollary}

Next we give a preliminary result which characterizes $J(u,v)$ of
(\ref{cost function}) from its associated BSDE. This result
generalizes the one by S. Hamad\`ene, and J. P. Lepeltier, 1995b
\cite{hamadene1995backward} (Theorem I.3). The main improvement is that
the drift of the diffusion, weak solution of (\ref{new equation of x}), is not bounded anymore but is instead of
linear growth.
\medskip

\begin{proposition}\label{Lemma BSDE solution existence wrt uv} Assume that Assumptions (A1), and (A2) on 
$f$, $h_i$, $g^i$, $i=1,2$, are fulfilled. Then for any pair $(u,
v)\in \cM$, there exists a pair of $\cP$-measurable processes $(W^{i,(u,v)}, Z^{i,(u,v)})$, $i=1,2$, with values in $\R\times
\R^m$ such that:
\begin{enumerate}
  \item[(i)] For any $q\geq 1$ and $i=1,2$, we have,
  \begin{equation}\label{estimate of w1uv and z1uv under puv}
  \be^{(u,v)}\Big[\sup_{0\leq s\leq T}|W_s^{i,(u,v)}|^q+\big(\int_0^T|Z_s^{i,(u,v)}|^2ds\big)^{\frac{q}{2}}\Big]<\infty.
  \end{equation}
\item[(ii)] For $t\leq T$, \begin{equation}\label{BSDE uv}
   \quad W_t^{i,(u,v)}= g^i(X_T^{0,x_0})+\int_t^T H_i(s, X_s^{0,x_0}, Z_s^{i,(u,v)}, u_s, v_s)ds-\int_t^T
  Z_s^{i,(u,v)}dB_s.\end{equation}
\end{enumerate}

\noindent The solution of BSDE (\ref{estimate of w1uv and z1uv under puv})-(\ref{BSDE uv}) is unique, moreover, $W_0^{i,(u,v)}= J^i(u,v)$ for $i=1,2$.
\end{proposition}

\begin{proof} We will give the proof for $i=1$ and of course it is similar for $i=2$. So for $(u,v)\in \m$, the process $(X_s^{0,x_0})_{s\leq T}$ is a weak solution of the
following stochastic differential equation:
\begin{equation*}
dX_s^{0,x_0}=f(s, X_s^{0,x_0}, u_s, v_s)ds+ \sigma(s,
X_s^{0,x_0})dB_s^{(u, v)}, \quad s\leq T \ \text{and}\  X_0^{0,x_0}= x_0,
\end{equation*}
\noindent where $B^{(u,v)}$ is a Brownian motion under $\bp^{(u,v)}$. Let us define the process $(W_t^{1,(u,v)})_{t\leq T}$ as follows:
\begin{equation}\label{define of w^uv}
\forall t\leq T,\,\,W_t^{1,(u,v)}\triangleq \be^{(u,
v)}\Big[g^1(X_T^{0,x_0})+\int_t^Th_1(s, X_s^{0,x_0}, u_s,
v_s)ds\mid \mathcal{F}_t\Big].
\end{equation}

Since the functions $g^1(x)$ and $h_1(s,x,u,v)$ are of
polynomial growth w.r.t $x$, we have, for any $r\geq 1$,
\begin{eqnarray}\label{W_uv is well defined}
&& \quad \be^{(u,v)}\Big[\abs{g^1(X_T^{0,x_0})+\int_0^T\abs{h_1(s,
X_s^{0,x_0},u_s, v_s)}ds}^{2r}\Big]\nonumber\\ && \qquad \leq
\be^{(u,v)}\Big[C\big(1+\sup_{s\leq T}\abs{X_s^{0,x_0}}^{2\gamma r}\big)\Big]\nonumber\\
&& \qquad \leq C(1+\abs{x_0}^{2\gamma r}).
\end{eqnarray}

\noindent The last inequality is due to estimate (\ref{estimate of X2}). Then (\ref{W_uv is well defined}) implies that the process
$(W_t^{1,(u,v)})_{t\leq T}$ is well defined.\medskip

\noindent Next for notation simplicity for any $t\leq T$, we denote by $\zeta_t$
the density function\\ $\zeta_t\left(\sigma^{-1}(s, X_s^{0,x_0})f(s, X_s^{0,x_0},u_s,
v_s)\right)$. Therefore,
\begin{align}\label{w_uv 3.9}
W_t^{1,(u,v)}&=(\zeta_t)^{-1}
\be \Big[\zeta_T\cdot\big(g^1(X_T^{0,x_0})+\int_t^Th_1(s,
X_s^{0,x_0}, u_s, v_s)ds\big)\mid \mathcal{F}_t\Big]\nonumber\\
&=(\zeta_t)^{-1}
\be\Big[\zeta_T\cdot\big(g^1(X_T^{0,x_0})+\int_0^Th_1(s, X_s^{0,x_0},
u_s, v_s)ds\big)\mid \mathcal{F}_t\Big]\\
&\quad -\int_0^th_1(s,
X_s^{0,x_0}, u_s, v_s)ds. \nonumber
\end{align}
\noindent By Corollary \ref{corol lp}, there exists
some $1<p_0<2$, such that $\zeta_T \in L^{p_0}(\R)$. Therefore, from Young's
inequality, we obtain, for any $1< \bar{q} <p_0$,
\begin{align*}
&\quad \be \Big[\Big|\zeta_T\cdot \big(g^1(X_T^{0,x_0})+\int_0^T
h_1(s, X_s^{0,x_0},u_s,v_s)ds\big)\Big|^{\bar{q}}\Big]\\
&\leq \frac{\bar{q}}{p_0}\be\left[\abs{\zeta_T}^{p_0}\right]+
\frac{p_0-\bar{q}}{p_0}\be\Big[\abs{g^1(X_T^{0,x_0})+\int_0^Th_1(s,
X_s^{0,x_0},u_s,v_s)ds}^{\bar{q}\cdot\frac{p_0}{p_0-\bar{q}}}\Big].
\end{align*}

\noindent Since $\frac{p_0}{\bar{q}}<2$, its conjugate
$\frac{p_0}{p_0-\bar{q}}>2$. Therefore, by the polynomial growth
assumptions of $g^1$ and $h_1$ w.r.t $x$ (Assumption
(A2)-(ii),(iii)) and the estimate (\ref{estimate of X1}), we have
\begin{equation*}
\be\Big[\abs{g^1(X_T^{0,x_0})+\int_0^Th_1(s,
X_s^{0,x_0},u_s,v_s)ds}^{\bar{q}\cdot\frac{p_0}{p_0-\bar{q}}}\Big]< \infty.
\end{equation*}
\noindent Then,
\begin{equation*}
\zeta_T\cdot (g^1(X_T^{0,x_0})+\int_0^T h_1(s, X_s^{0,x_0},u_s,
v_s)ds)\in L^{\bar{q}}.
\end{equation*}

\noindent Therefore thanks to the representation theorem, there exists a $\cP$-measurable and 
$\R^m$-valued process, 
$(\bar \theta_s)_{s\leq T}$ which satisfies,
\begin{equation*}
\be\Big[\big(\int_0^T
\left|\bar \theta_s\right|^2ds\big)^{\frac{\bar{q}}{2}}\Big]< \infty,
\end{equation*}
\noindent such that for any $t\le T$,
\begin{align*}
W_t^{1,(u,v)}&\!=\!
(\zeta_t)^{-1}\cdot\{\be \Big[\zeta_T\cdot\big(g^1(X_T^{0,x_0})+\int_0^T
h_1(s, X_s^{0,x_0},u_s, v_s)ds\big)\Big]+ \int_0^t
\bar \theta_sdB_s\}\\
&\quad -\int_0^th_1(s,
X_s^{0,x_0}, u_s, v_s)ds\\
&\triangleq (\zeta_t)^{-1}R_t-\int_0^th_1(s, X_s^{0,x_0},u_s, v_s)ds,
\end{align*}
\noindent where $$R_t\triangleq
\be \Big[\zeta_T\cdot\big(g^1(X_T^{0,x_0})+\int_0^T h_1(s,
X_s^{0,x_0},u_s, v_s)\big)\Big]+ \int_0^t \bar \theta_sdB_s, \quad
t\leq T.$$
But for any $s\le T$,
\begin{equation*}
d\zeta_s=\zeta_s\sigma^{-1}(s,X_s^{0,x_0})f(s,X_s^{0,x_0},u_s,v_s)dB_s.
\end{equation*}
Then by It\^o's formula we have
\begin{align*}
d(\zeta_s)^{-1}&=-(\zeta_s)^{-1}\cdot\Big\{\sigma^{-1}\left(s,X_s^{0,x_0}\right)f\left(s,X_s^{0,x_0},u_s,v_s\right)dB_s-\\
&\qquad|\sigma^{-1}\left(s,X_s^{0,x_0}\right)f\left(s,X_s^{0,x_0},u_s,v_s\right)|^2ds\Big\}, \quad s\leq T.
\end{align*}

\noindent Therefore, $\forall s\leq T$, 
\begin{align*}
dW_s^{1,(u,v)} &=-(\zeta_s)^{-1}\cdot \Big\{\sigma^{-1}(s,X_s^{0,x_0})f(s,X_s^{0,x_0},u_s,v_s)dB_s\\ &\quad -|\sigma^{-1}(s,X_s^{0,x_0})f(s,X_s^{0,x_0},u_s,v_s)|^2ds\Big\}R_s\\
&\quad +(\zeta_s)^{-1}\bar\theta_sdB_s+\big(-(\zeta_s)^{-1}\sigma^{-1}(s,X_s^{0,x_0})f(s,X_s^{0,x_0},u_s,v_s)\big)\bar\theta_sds\\&\quad -h_1(s,X_s^{0,x_0},u_s,v_s)ds.
\end{align*}

\noindent Next let us define
\begin{equation}\label{define of z^uv}
Z_s^{1,(u,v)}\triangleq
-(\zeta_s)^{-1}\Big\{R_s\sigma^{-1}(s,X_s^{0,x_0})f(s,X_s^{0,x_0},u_s,v_s)-\bar \theta_s\Big\}, \,s\leq T.
\end{equation}
Then it is easy to check that the pair of processes $(W_s^{1,(u,v)}, Z_s^{1,(u,v)})_{s\leq T}$ of (\ref{define of w^uv})-(\ref{define of z^uv}) satisfies the backward equation (\ref{BSDE uv}).
\medskip

We now focus on the estimates for the processes $(W_s^{1,(u,v)},
Z_s^{1,(u,v)})_{s\leq T}$. From (\ref{define of w^uv}), for any
$s\leq T$ and $q>1$,
\begin{equation*}
\sup_{0\leq t\leq T}\abs{W_t^{1,(u,v)}}^q = \sup_{0\leq t\leq
T}|\be^{(u,v)}\Big[g^1(X_T^{0,x_0})+\int_t^Th_1\left(s,
X_s^{0,x_0},u_s,v_s\right)ds\Big|\mathcal{F}_t\Big]|^q.
\end{equation*}

\noindent Then by conditional Jensen's inequality we have, 
\begin{equation*}
\sup_{0\leq t\leq T}\abs{W_t^{1,(u,v)}}^q \leq
\be^{(u,v)}\Big[\sup_{0\leq t\leq
T}\abs{g^1(X_T^{0,x_0})+\int_t^Th_1\left(s,X_s^{0,x_0},u_s,v_s\right)ds}^q\Big|\mathcal{F}_t\Big].
\end{equation*}
\noindent Therefore, since $g^1$ and $h_1$ are of polynomial growth, we have
\begin{align}\label{estimate of w1uv under puv}
\be^{(u,v)}\!\Big[\!\sup_{0\leq t\leq T}\abs{W_t^{1,(u,v)}}^q\Big]
&\leq\! \be^{(u,v)}\!\Big[\sup_{0\leq t\leq
T}\abs{g^1(X_T^{0,x_0})+\int_t^Th_1\left(s,X_s^{0,x_0},u_s,v_s\right)ds}^q\Big]\nonumber\\
&\leq \!C\be^{(u,v)}\!\Big[\sup_{0\leq t\leq
T}(1+\abs{X_t^{0,x_0}}^{\gamma q})\Big]\leq C(1+\abs{x_0}^{\gamma q})<\infty.
\end{align}
\noindent Next for each integer $k$, let us define:
$$\tau_k=\inf\{s\geq 0, \int_0^s |Z_s^{1,(u,v)}|^2ds\geq k\}\wedge T.$$
\noindent The sequence $(\tau_k)_k\geq 0$ is of stationary type and converges to $T$. By using It\^o's formula with $(W_{t\wedge\tau_k}^{1,(u,v)})^2$ we obtain: $\forall t\leq T$, 
\begin{align*}
|&W_{t\wedge\tau_k}^{1,(u,v)}|^2+\int_{t\wedge\tau_k}^{\tau_k}\abs{Z_s^{1,(u,v)}}^2ds\\
&= \!|W_{\tau_k}^{1,(u,v)}|^2\!+\!2\int_{t\wedge\tau_k}^{\tau_k}W_s^{1,(u,v)}h_1(s,
X_s^{0,x_0},u_s,v_s)ds\!-\!2\int_{t\wedge\tau_k}^{\tau_k}W_s^{1,(u,v)}Z_s^{1,(u,v)}dB_s^{(u,v)}.
\end{align*}

\noindent Thus, for $q>1$, taking the expectation of the power $\frac{q}{2}$ of the above equation on both sides and applying Young's
inequality, we see that there exists a constant $\underbar C$ such that, 
\begin{eqnarray}\label{3.19}
\be^{(u,v)}\Big[\big(\int_0^{\tau_k}\abs{Z_s^{1,(u,v)}}^2ds\big)^{\frac{q}{2}}\!\Big]
\!&\leq&\!
\underbar C\Big\{\!\be^{(u,v)}\!\left[|W_{\tau_k}^{1,(u,v)}|^q\right]\!+\!\be^{(u,v)}\!\Big[\!\big(\int_0^{\tau_k}\!\!\!\abs{W_s^{1,(u,v)}}^2ds\big)^{\frac{q}{2}}\Big]\nonumber\\
&\quad&
+\be^{(u,v)}\Big[\big(\int_0^{\tau_k}\abs{h_1(s,X_s^{0,x_0},u_s,v_s)}^2ds\big)^{\frac{q}{2}}\Big]\nonumber\\
&\quad& +\be^{(u,v)}\Big[\big|\int_0^{\tau_k}W_s^{1,(u,v)}Z_s^{1,(u,v)}dB^{(u,v)}_s\big|^{\frac{q}{2}}\Big]\Big\}.
\end{eqnarray}
\noindent Next taking into account the Assumptions (A2)-(ii) and estimate (\ref{estimate of w1uv under puv}), one can show that, 
\begin{align*}
&\be^{(u,v)}\Big[|W_{\tau_k}^{1,(u,v)}|^q\Big]+\be^{(u,v)}\Big[\big(\int_0^{\tau_k}\abs{W_s^{1,(u,v)}}^2ds\big)^{\frac{q}{2}}\Big]\\
&\qquad \qquad \qquad \qquad \qquad \qquad \qquad \qquad+\be^{(u,v)}\!\Big[\big(\int_0^{\tau_k}\abs{h_1(s,X_s^{0,x_0},u_s,v_s)}^2ds\big)^{\frac{q}{2}}\Big]\\
&\leq \bar C\Big \{\be^{(u,v)}\Big[\sup_{0\leq s\leq
\tau_k}\abs{W_s^{1,(u,v)}}^q\Big]+\be^{(u,v)}\Big[\sup_{0\leq
s\leq \tau_k}\big(1+\abs{X_s^{0,x_0}}^{2\gamma}\big)^{\frac{q}{2}}\Big]\Big \}\\
&\leq \bar C \Big\{\be^{(u,v)}\Big[\sup_{0\leq s\leq
T}\abs{W_s^{1,(u,v)}}^q\Big]+\{\be^{(u,v)}\Big[\sup_{0\leq s\leq
T}\left(1+\abs{X_s^{0,x_0}}^{2\gamma
q}\right)\Big]\}^{\frac{1}{2}}\Big\}<\infty.
\end{align*}
\noindent Meanwhile, it follows from the Burkholder-Davis-Gundy (BDG for short) that there exists a constant $C_q$, depending on $q$, such that
\begin{align*}
\be^{(u,v)}&\Big[\big|\int_0^{\tau_k}W_s^{1,(u,v)}Z_s^{1,(u,v)}dB^{(u,v)}_s\big|^{\frac{q}{2}}\Big]\\
&\leq C_q\be^{(u,v)}\Big[\big(\int_0^{\tau_k}\abs{W_s^{1,(u,v)}}^2\abs{Z_s^{1,(u,v)}}^2ds\big)^{\frac{q}{4}}\Big]\\
&\leq C_q\be^{(u,v)}\Big[\big(\sup_{0\leq s\leq \tau_k}\abs{W_s^{1,(u,v)}}\big)^{\frac{q}{2}}\big(\int_0^{\tau_k}\abs{Z_s^{1,(u,v)}}^2ds\big)^{\frac{q}{4}}\Big]\\
&\leq \frac{C_q^2\underbar{C}}{2}\be^{(u,v)}\Big[\big(\sup_{0\leq s\leq
T}\abs{W_s^{1,(u,v)}}\big)^{q}\Big]+\frac{1}{2\underbar{C}}\be^{(u,v)}\Big[\big(\int_0^{\tau_k}\abs{Z_s^{1,(u,v)}}^2ds\big)^{\frac{q}{2}}\Big],
\end{align*}
where $\underbar C$ is the one of (\ref{3.19}). Going back now to (\ref{3.19}) and using Fatou's Lemma, we conclude that for any $q>1$,
\begin{equation}\label{estimate of 1zuv under puv}
\be^{(u,v)}\Big[\big(\int_0^T\abs{Z_s^{1,(u,v)}}^2ds\big)^{\frac{q}{2}}\Big]<\infty.
\end{equation}
\noindent Estimates (\ref{estimate of w1uv under puv}) and
(\ref{estimate of 1zuv under puv}) yield to the conclusion
(\ref{estimate of w1uv and z1uv under puv}).

Finally note that, taking $t=0$ in (\ref{define of w^uv}) we obtain $W^{1,(u,v)}_0=J^1(u,v)$ since $ {\cal F}_0$ contains only $\bp$ and $\bp^{(u,v)}$ null sets since those probabilities are equivalent. 
\end{proof}

\begin{theorem}\label{bsde we mainly proof} Let us assume that:

\no (i) The Assumptions (A1), (A2) and (A3) are fulfilled ;

\no (ii) There exist two deterministic functions $\varpi^i(t,x)$, $i=1,2,$ with polynomial growth and two pairs of $\cP$-measurable processes $(W^i,Z^i)$, $i=1,2$, with values in $\R^{1+m}$ such that: For $i=1,2$,

(a) $\bp$-a.s., $\forall s\leq T$, $W^i_s=\varpi^i(s,X^{0,x}_s)$ and $Z^i(\omega):=(Z^i_t(\omega))_{t\leq T}$ is $dt$-square integrable ;

(b) For any $s\le T$,
\begin{equation}\label{main BSDE}
\left\{
\begin{aligned}
-dW_s^i&= H_i\left(s, X_s^{0,x}, Z_s^i, u^*(s, X_s^{0,x},Z_s^1, Z_s^2),
v^*(s, X_s^{0,x}, Z_s^1, Z_s^2)\right)ds- Z_s^i dB_s, \\
W_{T}^i&= g^i(X_T^{0,x}).
\end{aligned}
\right.
\end{equation}
Then the control $(u^*(s, X_s^{0,x}, Z_s^1, Z_s^2),v^*(s,
X_s^{0,x}, Z_s^1, Z_s^2))_{s\leq T}$ is admissible and a Nash equilibrium point for the NZSDG. 
\end{theorem}

\begin{proof}
For $s\leq T$, let us set $u^{*}_s= u^{*}(s, X_s^{0,x},Z_s^1, Z_s^2)$
and $v^{*}_s= v^{*}(s, X_s^{0,x},Z_s^1, Z_s^2)$, then $(u^*,v^*)\in \cM$. On the other hand we obviously have, thanks to Proposition \ref{Lemma BSDE solution existence wrt uv}, $W^1_0=J^1(u^*,v^*)$.

Next let $u$ be an arbitrary element of $\m_1$ and let us show that $W^1\leq W^{1,(u, v^{*})}$, which
yields $W^1_0=J^1(u^{*}, v^{*})\leq W^{1,(u, v^{*})}_0=J^1(u, v^*)$.
\medskip

\noindent The control $(u, v^*)$ is admissible and thanks to
Proposition \ref{Lemma BSDE solution existence wrt uv}, there exists
a pair of $\mathcal{P}$-measurable processes $(W^{1,(u, v^*)}, Z^{1,(u, v^*)})$  such that for any $q>1$,
\begin{equation}\label{eqjeu}\left\{
\begin{aligned}
&\be^{(u,v^*)}\Big[\sup_{0\leq t\leq T}|W_t^{1,(u,v^*)}|^q+\big(\int_0^T\abs{Z_s^{1,(u,v^*)}}^2ds\big)^{\frac{q}{2}}\Big]<\infty\\&
W_t^{1,(u, v^*)}\!\!=\!g^1(X_T^{0,x})\!+\!\!\int_t^T\! H_1(s, X_s^{0,x},
Z_s^{1,(u, v^*)}, u_s, v^*_s)ds\!-\!\!\int_t^T\! Z_s^{1,(u, v^*)}dB_s, \,\forall t\leq T.
\end{aligned}
\right.
\end{equation}

\noindent Afterwards, we aim to compare $W^1$ and $W^{1,(u, v^*)}$. So let us denote by
\begin{equation*}
\triangle W= W^1- W^{1,(u, v^*)} \mbox{ and }\triangle Z= Z^1- Z^{1,(u,
v^*)}.
\end{equation*}
\noindent For $k\geq 0$, we define the stopping time $\tau_k$ as
follows:
\begin{equation*}
\tau_k:= \inf\{s\geq 0, \abs{\triangle
W_s}+\int_0^s\abs{\triangle Z_r}^2dr\geq k\}\wedge T.
\end{equation*}
The sequence of stopping times $(\tau_k)_{k\geq 0}$ is of stationary type and converges to $T$. Next applying It\^{o}-Meyer formula to $\abs{(\triangle W)^+}^q$ $(q>1)$ (see Theorem 71, P.Protter, \cite{protter}, pp.221), between $t\wedge \tau_k$ and $\tau_k$, we obtain: $\forall t\leq T$, 
\begin{equation*}
\begin{aligned}
|(\triangle &W_{t\wedge \tau_k})^+|^q+c(q)\int_{t\wedge \tau_k}^{\tau_k}\abs{(\triangle W_s)^+}^{q-2}1_{\triangle W_s>0}\abs{\triangle Z_s}^2ds\\
&= \abs{(\triangle W_{\tau_k})^+}^q+ q\int_{t\wedge
\tau_k}^{\tau_k}\abs{(\triangle W_{s})^+}^{q-1}
1_{\triangle W_s>0} \Big(H_1(s, X_s^{0,x}, Z_s^1, u_s^*,
v_s^*)-\\
&\qquad H_1(s, X_s^{0,x}, Z_s^{(u, v^*)}, u_s, v^*_s)\Big)ds -
q\int_{t\wedge \tau_k}^{\tau_k}\abs{(\triangle
W_{s})^+}^{q-1}1_{\triangle W_s>0}\triangle Z_sdB_s\\
&=\abs{(\triangle W_{\tau_k})^+}^q + q\int_{t\wedge
\tau_k}^{\tau_k}\abs{(\triangle W_{s})^+}^{q-1}
1_{\triangle W_s>0} \Big(H_1(s, X_s^{0,x}, Z_s^1, u_s^*,
v_s^*)-\\
&\quad H_1(\!s,\! X_s^{0,x},\! Z_s^1, \!u_s, \!v^*_s) \!+\!H_1(s, X_s^{0,x}, Z_s^1, u_s, v^*_s)\!-\!H_1(s, X_s^{0,x}, Z_s^{(u,v^*)}, u_s, v^*_s)\Big)ds\\
&\quad - q\int_{t\wedge \tau_k}^{\tau_k}\abs{(\triangle
W_{s})^+}^{q-1}1_{\triangle W_s>0}\triangle Z_sdB_s,
\end{aligned}
\end{equation*}
where $c(q)=\frac{q(q-1)}{2}$. Considering now the generalized
Isaacs' Assumption (A3), we have that,
\begin{equation*}
H_1(s,X_s^{0,x}, Z_s^1, u_s^*,v_s^*)-H_1(s, X_s^{0,x}, Z_s^1, u_s, v^*_s)\leq 0, \forall s\leq T.
\end{equation*}

\noindent Therefore,
\begin{equation}\label{ito meyer formula}
\begin{aligned}
\big|(&\triangle W_{t\wedge \tau_k})^+\big|^q+c(q)\int_{t\wedge \tau_k}^{\tau_k}\abs{(\triangle W_s)^+}^{q-2}1_{\triangle W_s>0}\abs{\triangle Z_s}^2ds\\
&\leq \abs{(\triangle W_{\tau_k})^+}^q\!+\! q\int_{t\wedge
\tau_k}^{\tau_k}\abs{(\triangle W_{s})^+}^{q-1}
1_{\triangle W_s>0} \triangle Z_s \sigma^{-1}(s, X_s^{0,x})f(s, X_s^{0,x}, u_s, v^*_s)ds\\
&\quad - q\int_{t\wedge \tau_k}^{\tau_k}\abs{(\triangle
W_{s})^+}^{q-1}1_{\triangle W_s>0}\triangle Z_sdB_s\nonumber\\
&=\abs{(\triangle W_{\tau_k})^+}^q-q\int_{t\wedge
\tau_k}^{\tau_k}\abs{(\triangle
W_{s})^+}^{q-1}1_{\triangle W_s>0}\triangle Z_sdB_s^{(u,
v^*)},
\end{aligned}
\end{equation}
where $B^{(u, v^*)}= (B_t- \int_0^{t} \sigma^{-1}(s,
X_s^{0,x})f(s, X_s^{0,x}, u_s, v^*_s)ds)_{t\leq T}$ is an
$(\mathcal{F}_t^0, \bp^{(u, v^*)})$-Brownian motion. Then for any $t\leq T$,
\begin{equation}\label{3.11 middle result}
\begin{aligned}
\abs{(\triangle W_{t\wedge \tau_k})^+}^q\leq \abs{(\triangle W_{\tau_k})^+}^q- q\int_{t\wedge
\tau_k}^{\tau_k}\abs{(\triangle
W_{s})^+}^{q-1}1_{\triangle W_s>0}\triangle Z_sdB_s^{(u,
v^*)}.
\end{aligned}
\end{equation}

\noindent By definition of the stopping time $\tau_k$, we have
$$
\be^{(u, v^*)}\Big[\int_{t\wedge \tau_k}^{\tau_k}\abs{(\triangle W_s)^+}^{q-1}1_{\triangle W_s>0}\triangle Z_sdB_s^{(u,v^*)}\Big]=0.
$$Then taking expectation on both sides of (\ref{3.11 middle result}) we obtain:
\begin{align}\label{estim}
\be^{(u, v^*)}\left[\abs{(\triangle W_{t\wedge \tau_k})^+}^q\right] &\le  \be^{(u, v^*)}\Big[\abs{(W^1_{\tau_k}-W_{\tau_k}^{1,(u,v^*)})^+}^q\Big].
\end{align}
Next taking into account (\ref{eqjeu}) and the fact that $W^1$ has a representation through $\varpi^1$ which is deterministic and of polynomial growth
 and finally (\ref{estimate of X2}), we deduce that
\begin{equation}\label{unint}
\be^{(u, v^*)}\Big[\sup_{s\leq T}(|W^{1,(u,v^*)}_s|+|W^{1}_s|)^q\Big]<\infty.\end{equation}
As the sequence $((W^1_{\tau_k}-W_{\tau_k}^{1,(u,v^*)})^+)_k$
converges to $0$ as $k\rightarrow \infty$,  $\bp^{(u,v^*)}$-a.s., then it converges also to $0$ in $L^1(d\bp^{(u,v^*)})$ since it is uniformly integrable thanks to (\ref{unint}).
Taking now the limit w.r.t. $k$ on both sides of (\ref{estim}) and finally by Fatou's Lemma we deduce that:
 $$
 \be^{(u, v^*)}\left[\triangle W_{t}^+\right]=0,\,\,\forall t\leq T, 
 $$
which implies that $W^1\leq W^{1,(u, v^{*})}$, $\bp$-a.s., since the probabilities $\bp^{(u,v^*)}$ and $\bp$ are equivalent. Thus 
 $W^1_0=J^1(u^{*}, v^{*})\leq W^{1,(u, v^{*})}_0=J^1(u, v^*)$. 
 
In the same way one can show that if $v$ is an arbitrary element of $\cM_2$ then $W^2_0=J^2(u^{*}, v^{*})\leq W^{2,(u^*, v)}_0=J^2(u^*,v)$.
 Henceforth $(u^*,v^*)$ is a Nash equilibrium point for the NZSDG. 
\end{proof}

Now, the main emphasis is placed on the existence of a solution for
the BSDE (\ref{main BSDE}) with its properties.

\section{Existence of solutions for markovian  BSDEs related to differential games}
\subsection{Deterministic representation }
Let $\ell$ be an integer and let us consider $\bar f$ (resp. $\bar g$) a Borel measurable function from $[0,T]\times \R^{m+\ell+\ell\times m}$ (resp.
$\R^{m}$) into $\R^\ell$ (resp. $\R^\ell$) such that:

(a) For any fixed $(t,x)\in [0,T]\times \R^{m}$, the mapping $(y,z)\in \R^{\ell+\ell\times m}\longmapsto \bar f(t,x,y,z)$ is uniformly Lipschitz ;

(b) There exist real constants $C$ and $p>0$ such that
$$
 \abs{\bar {f}(t,x,y,z)}+\abs{\bar g(x)}\leq C(1+ \abs{x}^p), \ \forall (t,x,y,z)\in
[0, T]\times \R^{m+\ell+\ell\times m}.
$$
Then we have the following result by
El Karoui et al. \cite{el1997backward a} related to representation of solutions of BSDEs through deterministic functions in the Markovian framework of randomness.

\begin{proposition}\label{relationship of x and yz} Assume that (A1), (i)-(ii) and (A4) are fulfilled. Let $(t,x)\in \esp$ be fixed and $(\theta_s^{t,x})_{t\leq s\leq T}$ be the solution of SDE (\ref{forward equation of x}). Let $(y^{t,x}_s,z^{t,x}_s)_{t\leq s\leq T}$ be the solution of the following BSDE:
\begin{equation*}
\left\{
\begin{aligned}
&y^{t,x}\in \mathcal{S}_{t,T}^2(\R^\ell), z^{t,x}\in \mathcal{H}_{t,T}^2(\R^{\ell\times m});\\
&-dy_s^{t,x}= \bar{f}(s, \theta_s^{t,x}, y_s^{t,x}, z_s^{t,x})ds-
z_s^{t,x}dB_s,\,\, s\in[t,T];\\ &y_T= \bar g(\theta_T^{t,x}).
\end{aligned}
\right.
\end{equation*}
Then there exists a pair of measurable and deterministic
applications $\varpi$: $[0,T]\times \R^m \rightarrow \R^\ell$ and $\upsilon$:
$[0,T]\times \R^m \rightarrow \R^{\ell\times d}$ such that,
\begin{equation*}\bp- a.s., \forall t\leq s\leq T, \quad y_s^{t,x}= \varpi(s, \theta_s^{t,x})\ and \  z_s^{t,x}=
\upsilon(s, \theta_s^{t,x}). 
\end{equation*}
Moreover,
\\
  (i) $\forall (t, x)\in [0, T]\times \R^m$, $\varpi(t,x)= \be[\int_t^T{\bar {f}(r, \theta_r^{t,x}, y_r^{t,x}, z_r^{t,x})}dr+
  \bar g(\theta_T^{t,x})]$;\\
  (ii) For any other $(t_1, x_1)\in [0,T]\times \R^m$, the
  process $(\varpi(s, \theta_s^{t_1, x_1}),\  \upsilon(s, \theta_s^{t_1, x_1}))_{t_1\leq s\leq T}$ is the
  unique solution in $\mathcal{S}_{t_1,T}^2(\R^\ell)\times \mathcal{H}_{t_1,T}^2(\R^{\ell\times m})$ of the BSDE associated with the coefficients $(\bar {f}(s, \theta_s^{t_1,x_1}, y, z),\bar
  g(\theta_T^{t_1,x_1}))$ in the time interval $[t_1,T]$.$\Box$
\end{proposition}

We next recall the notion of domination which is important in order to show that equation (\ref{main BSDE}) has a solution.
\begin{definition}\label{cdtdom}\textbf{: $L^{q}$-Domination condition}
\\
Let $q\in ]1,\infty[$ be fixed. For a given $t_1\in [0,T]$, a family of probability measures
$\{\nu_1(s,dx), s\in [t_1, T]\}$ defined on $\R^m$ is said to be $L^{q}$-
dominated by another family of probability measures
$\{\nu_0(s,dx), s\in [t_1, T]\}$, if for any $\delta \in(0, T-t_1]$,
there exists an application $\phi^\d_{t_1}: [t_1+\d, T]\times \R^m \rightarrow
\R^+$ such that:
\\

(i) $\nu_1$(s, dx)ds= $\phi^\d_{t_1}$(s, x)$\nu_0$(s, dx)ds on $[t_1+\delta, T]\times$
  $\R^m$.
  
(ii) $\forall k\geq 1$, $\phi^\d_{t_1}(s,x) \in L^q([t_1+\delta, T]\times [-k, k]^m$; $\nu_0(s,
  dx)ds)$.$\Box$
\end{definition}\label{definition of lp dominated}
\medskip

We then have:

\begin{lemma}\label{aronson's estimate}
\no Assume (A1) and (A4) fulfilled and the drift term $b(t,x)$ of SDE 
(\ref{forward equation of x}) is bounded. Let $q\in ]1,\infty[$ be fixed, $(t_0,x_0)\in \esp$ and let $(\theta_s^{t_0,x_0})_{t_0\leq s\leq T}$
be the solution of SDE (\ref{forward equation of x}). Then for any $s\in (t_0,T]$, the law $\bar \mu(t_0, x_0; s, dx)$ of
$\theta_s^{t_0,x_0}$ has a density function $\rho_{t_0,x_0}(s,x)$, w.r.t. Lebesgue measure $dx$, which satisfies the following estimate: $\forall (s,x)\in (t_0,T]\times \R^m$, 
\begin{equation}\label{aron-est}
\varrho_1(s-t_0)^{-\frac{m}{2}}exp\left[-\frac{\Lambda\abs{x-x_0}^2}{s-t_0}\right]\leq
\rho_{t_0,x_0}(s,x)\leq
\varrho_2(s-t_0)^{-\frac{m}{2}}exp\left[-\frac{\lambda\abs{x-x_0}^2}{s-t_0}\right]
\end{equation}
\no where $\varrho_1$, $\varrho_2$, $\Lambda$, $\lambda$ are real constants such that $0\leq \varrho_1 \leq \varrho_2$ and $0\leq \lambda \leq \Lambda$. Additionally for any $(t_1,x_1)\in [t_0,T]\times \R^m$, the family of laws $\{\bar\mu(t_1,x_1;s,dx), s\in [t_1,T]\}$ is
$L^q$-dominated by $\bar\mu(t_0,x_0;s,dx)$.
\end{lemma}
\no $Proof$: Since $\sigma$ satisfies (\ref{horm}) and $b$ is bounded, then by Aronson's result (see \cite{aronson1967bounds}),
the law $\bar\mu(t_0, x_0; s, dx)$ of
$\theta_s^{t_0,x_0}$, $s\in ]t_0,T]$, has a density function $\rho_{t_0,x_0}(s,x)$ which satisfies estimate (\ref{aron-est}). 
\ms

Let us focus on the second claim of the lemma. Let $(t_1,x_1)\in [t_0,T]\times \R^m$ and $s\in (t_1,T]$. Then, \begin{equation*}
\rho_{t_1,x_1}(s,x)=[\rho_{t_1,x_1}(s,x)\rho^{-1}_{t_0,x_0}(s,x)]\rho_{t_0,x_0}(s,x)=\phi_{t_1,x_1}(s,x)\rho_{t_0,x_0}(s,x)
\end{equation*}
\no with
$$\phi_{t_1,x_1}(s,x)=\left[\rho_{t_1,x_1}(s,x)\rho^{-1}_{t_0,x_0}(s,x)\right],
(s,x)\in (t_1,T]\times \R^m.$$ 
For any $\d\in (0,T-t_1]$, $\phi_{t_1,x_1}$ is defined on $[t_1+\d,T]$. Moreover for any $(s,x)\in [t_1+\d,T]\times \R^m$ it holds,
$$\begin{array}{ll}
\bar\mu(t_1,x_1;s,dx)ds&=\rho_{t_1,x_1}(s,x)dxds\\{}&=\phi_{t_1,x_1}(s,x)
\rho_{t_0,x_0}(s,x)dxds\\{}&=\phi_{t_1,x_1}(s,x)\bar\mu(t_0,x_0;s,dx)ds.
\end{array}$$
Next by (\ref{aron-est}), for any $(s,x)\in [t_1+\d,T]\times \R^m$,
\begin{equation*} 
0\le \phi_{t_1,x_1}(s,x)\leq
\frac{\varrho_2(s-t_1)^{-\frac{m}{2}}}{\varrho_1(s-t_0)^{-\frac{m}{2}}}exp\left[\frac{\Lambda\abs{x-x_0}^2}{s-t_0}-\frac{\lambda\abs{x-x_1}^2}{s-t_1}\right]:=\Phi_{t_1,x_1}(s,x).
\end{equation*}
It follows that for any $k\geq 0$, the function $\Phi_{t_1,x_1}(s,x)$ is bounded on $[t_1+\d,T]\times [-k,k]^m$ by a constant $\kappa$ which depends on $t_0$, $t_1$, $\d$, $\Lambda$, $\lambda$, $k$ and $x_0$. Next let $q\in (1,\infty)$, then,
$$\begin{array}{ll}
\int_{t_1+\d}^T\int_{[-k,k]^m}\Phi_{t_1,x_1}(s,x)^q\bar\mu(t_0,x_0;s,dx)ds&\leq \kappa^q \int_{t_1+\d}^T\int_{[-k,k]^m}\bar\mu(t_0,x_0;s,dx)ds\\{}&= \kappa^q\int_{t_1+\d}^Tds\be[1_{[-k,k]^m}(\theta_s^{t_0,x_0})]\leq \kappa^q T.
\end{array}
$$
Thus $\Phi$ and then $\phi$ belong to $L^q([t_1+\delta, T]\times [-k, k]^m$; $\nu_0(s,
  dx)ds)$. It follows that the family of measures 
$\{\bar\mu(t_1,x_1;s,dx), s\in [t_1,T]\}$ is
$L^q$-dominated by $\bar\mu(t_0,x_0;s,dx)$.  $\Box$
\ms 
  
As a by-product we have:
\begin{corollary}\label{cor dom}
Let $x_0\in \R^m$, $\tx$, $s\in (t,T]$ and 
$\mu(t,x;s,dy)$ the law of $X^{t,x}_s$, \textit{i.e.}, 
$$\forall A \in \mathcal{B}(\R^m),\,\,
\mu(t,x;s,A)= \bp(X_s^{t,x}\in A).$$ Under (A1) on $\sigma$, for any $q\in (1,\infty)$, the family 
of laws $\{\mu(t,x;s,dy), s\in [t,T]\}$ is $L^q$-dominated by 
$\{\mu(0,x_0;s,dy), s\in [t,T]\}$.$\Box$
\end{corollary} 
\subsection{The main result}
We are now ready to provide a solution for BSDE (\ref{main BSDE}) which satisfies the representation property via deterministic functions with polynomial growth.

\begin{theorem}\label{th main result section 4}
Let $x_0\in \R^m$ be fixed. Then under the Assumptions (A1), (A2) and (A3), there exist:

(i) Two pairs of $\mathcal{P}$-measurable processes $(W^i_s,Z^i_s)_{s\leq T}$, $i=1,2$, such that: $\forall i\in \{1,2\}$,
\begin{equation}\label{main BSDE starting from 0 a}
\left\{
\begin{aligned}
&\bp-a.s., \,\,Z^i(\omega)=(Z^i_s(\omega))_{s\leq T} \mbox{ is }dt-\mbox{square integrable };\\
&\!-\!dW_s^i\!=\! H_i\left(s, X_s^{0,x_0}, Z_s^i, u_1^*(s, X_s^{0,x_0},
Z_s^1, Z_s^2),
u_2^*(s, X_s^{0,x_0}, Z_s^1, Z_s^2)\right)ds- Z_s^i dB_s, \\
&\qquad \qquad \qquad \qquad \qquad \qquad \qquad \qquad \qquad\qquad \qquad \qquad \qquad \qquad \qquad  \forall \,s\leq T\,\,;\\
&W_{T}^i= g^i(X_T^{0,x_0}).
\end{aligned}
\right.
\end{equation}

(ii) Two measurable deterministic functions
$\varpi^i$, $i=1,2$ with polynomial growth defined from $\esp$ into $\R$ such that:
$$\forall i=1,2,\,\,
W^i_s=\varpi^i(s,X^{0,x_0}_s), \,\,\forall s\in [0,T].$$
\end{theorem}
\begin{proof} \quad It will be divided into
five steps. We first construct an approximating sequence of BSDEs which have
solutions according to Proposition \ref{relationship of x and yz}, we then provide a priori estimates of the solutions of those BSDEs. Finally 
we prove that those solutions are convergent (at least for a subsequence) and the limit is a solution for the BSDE (\ref{main BSDE starting from 0 a}). 
\medskip

\noindent\textbf{Step 1}: Construction of the approximating sequence
$(W_s^{in;(t,x)}, Z_s^{in;(t,x)})_{s\le T}$, $n\geq 1$, ${i=1,2}$.\medskip

\noindent Let $\xi$ be an element of $C^{\infty}(\R^{2m},\R)$ with
compact support and satisfying
$$\begin{array}{l}
\int_{\R^{2m}}\xi(y,z)dydz=1.\end{array}
$$
\noindent For $n\geq 1$, $i=1,2$ and $(t,x,z^1,z^2)\in
[0,T]\times \R^{3m}$, we set
\begin{align*}
\underbar{H}&_i^{n}\left(t,
x,z^1,\left(u_1^*,u_2^*\right)\left(t,x,z^1,z^2\right)\right)\\
&=\int_{\R^{2d}}n^2H_i\left(t,\varphi_n(x),y,\left(u_1^*,u_2^*\right)\left(t,\varphi_n(x),y,z\right)\right)\xi\left(n\left(z^1-y\right),n\left(z^2-z\right)\right)dydz,
\end{align*}
\noindent where for any $x=(x_j)_{1\le j\le m}\in \R^m$, $\varphi_n(x):=((x_j\vee (-n))\wedge
n)_{1\le j\le m}.$
\medskip

\noindent We next define $\psi\in C^{\infty}(\R^{2m},\R)$ satisfying 
\begin{equation*}
\psi(y,z)=\left\{
\begin{aligned}
1\mbox{ if }\quad \abs{y}^2+\abs{z}^2\leq 1,\\
0\mbox{ if }\quad \abs{y}^2+\abs{z}^2\geq 4.
\end{aligned}
\right.
\end{equation*}
\noindent Then, for $i=1,2$, the measurable functions $H_i^{n}$,
$n\geq 1$, defined as follows:
\begin{align*}
H_i^{n}\left(t,x,z^1,z^2\right):=
\psi\left(\frac{z^1}{n},\frac{z^2}{n}\right)&\underbar{H}_i^{n}\left(t,x,z^1,\left(u_1^*,u_2^*\right)\left(t,x,z^1,z^2\right)\right),\\
&(t,x,z^1,z^2)\in
[0,T]\times \R^{3m}, \,\,
\end{align*}
\noindent satisfy the following properties:
\begin{enumerate}
  \item [(a)] $H_i^{n}$ is uniformly Lipschitz w.r.t $(z^1,z^2)$ ;
  \item [(b)] $\abs{H_i^{n}\left(t,x,z^1,z^2\right)}\leq
  C\left(1+\abs{\varphi_n(x)}\right)
  \abs{z^i}+C\left(1+\abs{\varphi_n(x)}^{\gamma}\right)$.
  \item [(c)]
  $\abs{H_i^{n}\left(t,x,z^1,z^2\right)}\leq
  c_n$, for any $(t, x,z^1,z^2)$.
  \item [(d)]For any $(t,x)\in [0,T]\times \R^m$ and $K$ a
  compact subset of $\R^{2m}$,
  \begin{align*}
  \sup_{(z^1,z^2)\in
  K}&\!\!|H_i^{n}\left(t,x,z^1,z^2\right)\!-\!H_i\left(t,x,z^1,\left(u_1^*,u_2^*\right)\left(t,x,z^1,z^2\right)\right)|\!\rightarrow\!
  0, \text{as}\  n\rightarrow \infty.
  \end{align*}
\end{enumerate}

Let us notice that (b) is valid since $u_1^*$ and $u_2^*$ take their values in compact sets.  

The constant $\gamma$, which we can choose greater than 1, is the one of polynomial growth of $h_i$, $i=1,2$. 
Therefore, from points $(a)$ and (c) and Proposition \ref{relationship of x and yz}, for each $n\geq 1$, $i=1,2$ and $(t,x)\in \esp$, there exist solutions $(W_s^{in;(t,x)},
Z_s^{in;(t,x)})_{s\leq T}$ in $\mathcal{S}_{t,T}^2(\R)\times
\mathcal{H}_{t,T}^2(\R^m)$ such that for any $s\in [t,T]$, 
\begin{equation}\label{4.4}
W_s^{in;(t,x)}= g^i(X_T^{t,x})+\int_s^T H_i^n(r, X_r^{t,x}, Z_r^{1n;(t,x)},
Z_r^{2n;(t,x)})dr- \int_s^T Z_r^{in;(t,x)} dB_r.
\end{equation}
\noindent Then, by Proposition \ref{relationship of x and yz}, for
$i=1,2$, there exists a sequence of measurable deterministic functions $\varpi^{in}$: $[0,T]\times \R^m \rightarrow \R$ and $\upsilon^{in}$:
$[0,T]\times \R^m \rightarrow \R^m$ such that
\begin{equation}\label{uin vin}
\forall s\in [t,T],\,\,
W_{s}^{in;(t,x)}= \varpi^{in}(s, X_s^{t,x}) \quad \text{and}\quad Z_s^{in;(t,x)}=
\upsilon^{in}(s, X_s^{t,x}).
\end{equation}
\noindent Moreover, for $i=1,2$ and $n\geq 1$, $\varpi^{in}$ satisfies
\begin{equation*}
\varpi^{in}(t,x)=\be\Big[g^i(X_T^{t,x})+\int_t^TF^{in}
(s,X_s^{t,x})ds\Big],\quad
\forall (t,x)\in [0,T]\times \R^m,
\end{equation*}
\noindent with 
\begin{equation*}
F^{in}(t,x)= H_i^n\left(t,x,\upsilon^{1n}(t,x),\upsilon^{2n}(t,x)\right), \,\,(t,x)\in [0,T]\times \R^m.
\end{equation*}

\noindent\textbf{Step 2}: The deterministic functions $\varpi^{in}$ are of polynomial growth uniformly w.r.t. $n$, i.e., there exist two constants $C$ and $\lambda$ such that for any $n\ge 1$, $i=1,2$, 
\begin{equation}\label{growthu}\forall (t,x)\in \esp,\,\,
|\varpi^{in}(t,x)|\leq C(1+|x|^\lambda).
\end{equation}

We will deal with the case of index $i=1$, the case of $i=2$ can be treated in the same way. 
For each $n\geq 1$, let us consider the following BSDE: $\forall s\in [t,T],$
\begin{equation}\label{bsde tilde}
\left\{\begin{array}{l}
(\bar{W},\bar{Z})\in \mathcal{S}_{t,T}^2(\R)\times
\mathcal{H}_{t,T}^2(\R^m)\,;\\
\bar{W}_s^{1n}\!=\! g^1(X_T^{t,x})\!+\!\int_s^T C(1+|\varphi_n(X_r^{t,x})|)|\bar{Z}_r^{1n}|\!+\!C(1+\abs{\varphi_n(X_r^{t,x})}^{\gamma})dr\!-\! \int_s^T \bar{Z}_r^{1n}dB_r.\end{array}
\right.
\end{equation}
For any $x \in \R^m$ and $n\ge 1$, the function \\$z^1\in \R^m \longmapsto C(1+|\varphi_n(X^{t,x}_s)|)|z^1|+C(1+|\varphi_n(X^{t,x}_s)|^{\gamma})$ is Lipschitz continuous. Then the solution $(\bar{W}^{1n},\bar{Z}^{1n})$ exists and is unique. Moreover through an adaptation of the result given in (El Karoui and al.,1997,\cite{el1997backward a}), we can infer the existence of deterministic measurable function $\bar{\varpi}^{1n}$: $[0,T]\times
\R^m \rightarrow \R$ such that, for any $s\in[t,T]$,
\begin{equation}\label{estimubarn}
\bar{W}_s^{1n}= \bar{\varpi}^{1n}(s, X_s^{t,x}).
\end{equation}
Next let us consider the process 
$$B^n_s= B_s- \int_0^s 1_{[t,T]}(r)C(1+|\varphi_n(X_r^{t,x})|)\mbox{sign}(\bar{Z}_r^{1n})dr,\,\, 0\leq s\leq T,$$
which is, thanks to Girsanov's Theorem, a Brownian motion under the probability $\bP^n$ on $(\Omega,
\mathcal{F})$ whose density with respect to $\bP$ is
$$\zeta_T\{C(1+|\varphi_n(X_s^{t,x})|)\mbox{sign}(\bar{Z}_s^{1n})1_{[t,T]}(s)\},$$ 
where for any $z=(z^i)_{i=1,...,d}\in \R^m$, $\mbox{sign}(z)=(1_{[|z^i|\neq 0]}\frac{z^i}{|z^i|})_{i=1,...,d}$ and  $\zeta_T(\cdot)$ is
defined by (\ref{density function}). Then (\ref{bsde tilde}) becomes
\begin{equation*}
\bar{W}_s^{1n}= g^1(X_T^{t,x})+\int_s^T
C(1+|\varphi_n(X_r^{t,x})|^{\gamma})dr-
\int_s^T \bar{Z}_r^{1n} dB^n_r, \quad 0\leq s\leq T.
\end{equation*}
\noindent Therefore, taking into 
account of (\ref{estimubarn}), we deduce,
\begin{equation*}
\bar{\varpi}^{1n}(t,x)= \be^n\Big[g^1(X_T^{t,x})+\int_t^T
C\left(1+\abs{\varphi_n(X_s^{t,x})}^{\gamma}\right)ds|\mathcal{F}_t\Big],
\end{equation*}

\noindent where $\be^n$ is the expectation under probability $\bP^n$.
Taking the expectation on both sides under the probability $\bP^n$
and considering $\bar{\varpi}^{1n}(t,x)$ is deterministic, one obtains,
\begin{equation*}
\bar{\varpi}^{1n}(t,x)= \be^n\Big[g^1(X_T^{t,x})+ \int_t^T
C\left(1+\abs{\varphi_n(X_s^{t,x})}^{\gamma}\right)ds\Big].
\end{equation*}
\noindent Then by the Assumption (A2)-(iii) we have: $\forall (t,x)\in \esp$, 
\begin{equation*}
\begin{aligned}
\abs{\bar{\varpi}^{1n}(t,x)}&\leq C\be^n\Big[\sup_{0\leq s\leq T}\left(1+ \abs{X_s^{t,x}}^{\gamma}\right)\Big]\\
        &= C\be\Big[\big(\sup_{0\leq s\leq T}\left(1+ \abs{X_s^{t,x}}^{\gamma}\right)\big)\Big(\zeta_T(C\left(1+\left|\varphi_n\left(X_s^{t,x}\right)\right|\right)\mbox{sign}(\bar{Z}_s^{1n}))\Big)\Big]. \end{aligned}
\end{equation*}

\noindent By Lemma \ref{density function lp bounded}, there exists
some $1<p_0<2$ (which does not depend on $(t,x)$), such that,
\begin{equation}\label{4.5}
\be\Big[\Big(\zeta_T(C\left(1+\left|\varphi_n\left(X_s^{t,x}\right)\right|\right)\mbox{sign}(\tilde{Z}_s^{1n}))\Big)^{p_0}\Big]<\infty.
\end{equation}

\noindent Then thanks to Young's inequality, we obtain,
\begin{align*}
&\abs{\bar{\varpi}^{1n}(t,x)}\\
&\leq C\{\be\Big[\sup_{0\leq s\leq T}\left(1+
        {\abs{X_s^{t,x}}}^{\gamma}\right)^{\frac{p_0}{p_0-1}}\Big]+\be\Big[\big(\zeta_T(C\left(1+\left|\varphi_n\left(X_s^{t,x}\right)\right|\right)\mbox{sign}(\tilde{Z}_s^{1n}))\big)^{p_0}\Big]\}.
\end{align*}
Finally estimates (\ref{4.5}) and (\ref{estimate of X}) yield
that,
\begin{equation*}
\abs{\bar{\varpi}^{1n}(t,x)}\leq C(1+ \abs{x}^{\lambda}),
\end{equation*}

\noindent where $\lambda= \frac{\gamma p_0}{p_0-1}>2$. Next taking into account point (b) and using comparison of solutions of BSDEs (\cite{el1997backward a}, pp.23) we obtain for any $s\in [t,T]$,
$$\bar W^{1n}_s=\bar \varpi^{1n}(s, X^{t,x}_s)\ge 
W^{1n;(t,x)}_s=\varpi^{1n}(s, X^{t,x}_s), \forall s\in [t,T],
$$
and then, choosing $s=t$, we get that $\varpi^{1n}(t,x)\leq C(1+ \abs{x}^{\lambda})$, $(t,x)\in \esp$. But in a similar way one can show that for any $(t,x)\in \esp$, \\$\varpi^{1n}(t,x)\geq -C(1+ \abs{x}^{\lambda})$. Therefore $\varpi^{1n}$ is of polynomial growth w.r.t. $(t,x)$ uniformly in $n$, i.e., it satisfies (\ref{growthu}).$\Box$
\ms

\noindent\textbf{Step 3}: There exists a constant $C$ independent of $n$ and 
$t,x$ such that for any $t\leq T$, for $i=1,2$,
\begin{equation}\label{estimate of Z}
\be[\int_t^T|Z_s^{in;(t,x)}|^2ds]\le C.
\end{equation}
Actually, we obtain from estimates (\ref{growthu}) and (\ref{estimate of X1}) that for any $\alpha\ge 1$, $i=1,2$,
$$
\be [\sup_{t\leq s\leq T}|W^{in;(t,x)}_s|^\alpha]\leq C.
$$
\noindent Going back to equation (\ref{4.4}) and making use of It\^o's formula with
$(W_s^{1n;(t,x)})^2$, we obtain, in a standard way, the result (\ref{estimate of Z}). The proof is omitted for conciseness.$\Box$
\bs

\noindent \textbf{Step 4}: There exists a subsequence of $((W_s^{in;(0,x_0)},
Z_s^{in;(0,x_0)})_{s\in [0,T]})_{n\geq 1}$, $i=1,2$, which converges respectively to 
$(W_s^i, Z_s^i)_{0\leq s\leq
T}$, $i=1,2$, solution of the BSDE (\ref{main BSDE starting from 0 a}).
\smallskip

\no Actually for $i=1,2$ and $n\geq 1$, by (\ref{uin vin}) we know that 
$$
W_s^{in;(0,x_0)}=\varpi^{in}(s, X^{0,x_0}_s), s\leq T 
$$ where the deterministic functions $\varpi^{in}$ verifies:  
\begin{equation}\label{main th u^1n}\forall (t,x)\in \esp,\,\,
\varpi^{in}(t,x)= \be\Big[g^i(X_T^{t,x})+\int_t^T
F^{in}(s,
X_s^{t,x})ds\Big].
\end{equation}

\noindent Let us now fix $q\in (1,2)$. Taking account of point $(b)$, we have: 
\begin{align*}
\be\Big[\int_0^T\abs{F^{in}&\left(s, X_s^{0,x_0}\right)}^qds\Big]=\int_{\esp}\abs{F^{in}(s,y)}^q\mu(0,x_0;s,dy)ds\\
&\leq
C\be\Big[\int_0^T\abs{Z_s^{in;(0,x_0)}}^q\left(1+\abs{X_s^{0,x_0}}^q\right)+
\left(1+\abs{X_s^{0,x_0}}^{\gamma
q}\right)ds\Big].
\end{align*}
\noindent By H\"older and Young's inequalities, one can
show that,
\begin{align}\label{l^p bound of F^1n}
\be\Big[\int_0^T\abs{F^{in}&\left(s, X_s^{0,x_0}\right)}^qds\Big]\nonumber\\
&\leq
C\be\Big[\big(\int_0^T\abs{Z_s^{in;(0,x_0)}}^2ds\big)^{\frac{q}{2}}\big(\int_0^T\left(1+
\abs{X_s^{0,x_0}}\right)^{\frac{2q}{2-q}}ds\big)^{\frac{2-q}{2}}\Big]\nonumber\\
&\quad +C\be\Big[\int_0^T\left(1+\abs{X_s^{0,x_0}}^{\gamma q}\right)ds\Big]\nonumber\\
&\leq
C\{\be\Big[\int_0^T\abs{Z_s^{in;(0,x_0)}}^2ds\Big]+
\be\Big[1+\sup_{0\le s \le T}|X^{0,x_0}_s|^\theta]\Big]\}
\end{align} for constant $\theta=(\gamma q)\vee\frac{2q}{2-q}$ which is greater than 2 with $1<q<2$ and $\gamma>1$. Taking now into account the estimates 
(\ref{estimate of Z}) and (\ref{estimate of
X1}) we deduce that, $$\be\Big[\int_0^T\abs{F^{in}\left(s, X_s^{0,x_0}\right)}^qds\Big]=\int_{\esp}\abs{F^{in}(s,y)}^q\mu(0,x_0;s,dy)ds\leq
C.
$$As a result, there exists a
subsequence $\{n_k\}$ (for notational simplification, we still denote
it by \{$n$\}) and two $\mathcal{B}([0,T])\otimes \mathcal{B}(\R^m)$-measurable deterministic functions $F^i(s,x)$, $i=1,2$, such that,
\begin{equation}\label{F^1n is lp weakly convergence}
F^{in}\rightarrow F^i \quad \text{weakly in}\quad L^q([0,T]\times
\R^m; \mu(0,x_0;s,dx)ds).\end{equation}

Next we aim to prove that $(\varpi^{in}(t,x))_{n\geq 1}$ is a Cauchy
sequence for each $(t,x)\in \esp$, i=1,2. 

\noindent So let $(t,x)$ be fixed, $\delta>0$, $k$,
$n$ and $m\geq1$ be integers. From (\ref{main th u^1n}), we have,
\begin{equation*}
\begin{aligned}
\abs{\varpi^{in}(t,x)-\varpi^{im}(t,x)}&=
\Big|\be\Big[\int_t^T\left[F^{in}(s,X_s^{t,x})-F^{im}(s,X_s^{t,x})\right]ds\Big]\Big|\\
&\leq
\Big|\be\Big[\int_t^{t+\delta}\left[F^{in}(s,X_s^{t,x})-F^{im}(s,X_s^{t,x})\right]ds\Big]\Big|\\
&\quad +\Big|\be\Big[\int_{t+\delta}^T\left(F^{in}(s,X_s^{t,x})-F^{im}(s,X_s^{t,x})\right).1_{\{\abs{X_s^{t,x}}\leq
k\}}ds\Big]\Big|\\
&\quad +\Big|\be\Big[\int_{t+\delta}^T\left(F^{in}(s,X_s^{t,x})-F^{im}(s,X_s^{t,x})\right).1_{\{\abs{X_s^{t,x}}>
k\}}ds\Big]\Big|.
\end{aligned}
\end{equation*}
On the right side, according to (\ref{l^p bound of
F^1n}), we have,
\begin{equation*}
\begin{aligned}
\be\Big[\int_t^{t+\delta}&\abs{F^{in}(s, X_s^{t,x})- F^{im}(s,
X_s^{t,x})}ds\Big]\\
&\leq
\delta^{\frac{q-1}{q}}\Big\{\be\Big[\int_t^T\abs{F^{in}(s,X_s^{t,x})-F^{im}(s,X_s^{t,x})}^qds\Big]\Big\}^{\frac{1}{q}}\\
&\leq C\delta^{\frac{q-1}{q}}.
\end{aligned}
\end{equation*}

\noindent At the same time, thanks to Corollary \ref{cor dom} associated to $L^{\frac{q}{q-1}}$-domination implies that:
\begin{equation*}
\begin{aligned}
\big|\be\big[\int_{t+\delta}^T&\left(F^{in}(s,X_s^{t,x})-F^{im}(s,X_s^{t,x})\right).1_{\{\abs{X_s^{t,x}}\leq
k\}}ds\big]\big|\\
&=
\big|\int_{\R^m}\int_{t+\delta}^T(F^{in}(s,\eta)-F^{im}(s,\eta)).1_{\{\abs{\eta}\leq
k\}}\mu(t,x;s,d\eta)ds\big|\\
&=
\big|\int_{\R^m}\int_{t+\delta}^T(F^{in}(s,\eta)-F^{im}(s,\eta)).1_{\{\abs{\eta}\leq
k\}}\phi_{t,x}(s,\eta)\mu(0,x_0;s,d\eta)ds\big|.
\end{aligned}
\end{equation*}

\noindent Since $\phi_{t,x}(s,\eta) \in L^{\frac{q}{q-1}}([t+\delta,
T]\times [-k, k]^m$; $\mu(0,x_0; s, d\eta)ds)$, for $k\geq 1$, it
follows from (\ref{F^1n is lp weakly convergence}) that for each $(t,x)\in \esp$, we have,
\begin{equation*}
\be\Big[\int_{t+\delta}^T\left(F^{in}(s, X_s^{t,x})-F^{im}(s,
X_s^{t,x})\right)1_{\{\abs{X_s^{t,x}}\leq
k\}}ds\Big]\rightarrow 0 \mbox{ as }n,m\rw \infty.
\end{equation*}
\noindent Finally,
\begin{equation*}
\begin{aligned}
\big|\be\big[&\int_{t+\delta}^T\left(F^{in}(s,X_s^{t,x})-F^{im}(s,X_s^{t,x})\right).1_{\{\abs{X_s^{t,x}}>
k\}}ds\big]\big|\\
&\leq C\Big\{\be\Big[\int_{t+\delta}^T
1_{\left\{\abs{X_s^{t,x}}>k\right\}}ds\Big]\Big\}^{\frac{q-1}{q}}\Big\{\be\Big[\int_{t+\delta}^T\abs{F^{in}(s,X_s^{t,x})-F^{im}(s,X_s^{t,x})}^qds\Big]\Big\}^{\frac{1}{q}}\\
&\leq Ck^{-\frac{q-1}{q}}
\end{aligned}
\end{equation*}

Therefore, $(\varpi^{in}(t,x))_{n\geq 1}$ is a Cauchy sequence for each $(t,x)\in \esp$ and then there
exists a measurable application $\varpi^i$ on $[0,T]\times \R^m$, $i=1,2,$ such that
for each $(t,x)\in \esp$, $i=1,2,$
\begin{equation*}
\lim_{n\rightarrow \infty}\varpi^{in}(t,x)=\varpi^i(t,x).
\end{equation*}
Additionally we obtain from estimate 
(\ref{growthu}) that $\varpi^i$ is of polynomial growth, i.e., 
\begin{equation}\label{croipoly}\forall (t,x)\in \esp,\,
|\varpi^i(t,x)|\leq C(1+|x|^\lambda).
\end{equation}
Therefore for any $ t\geq 0$.
\begin{equation*}
\lim_{n\rightarrow \infty} W_t^{in;(0,x_0)}(\omega)= \varpi^i(t,X_t^{0,x_0}(\omega)) \mbox{ , }\abs{W_t^{in;(0,x_0)}(\omega)}\leq C(1+|X_t^{0,x_0}(\omega)|^{\lambda}),\,\,\bp-a.s.
\end{equation*}
By Lebesgue's dominated convergence theorem,
$((W_t^{in;(0,x_0)})_{t\leq T})_{n\geq 1}$ converges to $W^i=(\varpi^i(t, X_t^{0,x_0}))_{t\leq
T}$ in $\mathcal{H}_T^{\kappa}(\R)$ for any $\kappa
\geq 1$, that is, for $i=1,2$,
\begin{equation}\label{w1n hk cov}
\be\Big[\int_0^T \abs{W^{in;(0,x_0)}_s-W_s^i}^\kappa ds\Big]\rightarrow 0,
\qquad \text{as}\ n\rightarrow \infty.
\end{equation}

We next show that $(W^{in;(0,x_0)})_{n\geq 0}$ is convergent in $\mathcal{S}^2_T(\R)$, $i=1,2$, as well. But first let us show that for $i=1,2$, the sequence $(Z^{in;(0,x_0)}=(\upsilon^{in}(t, X_t^{0,x_0}))_{t\leq
T})_{n\geq 1}$ has a limit in $\mathcal{H}_T^2(\R^m)$. As usual, we only deal with the case $i=1$. For $n,m \geq 1$ and $s\leq T$, using It\^{o}'s formula
with $(W_s^{1n}-W_s^{1m})^2$ (we omit the subscript $(0,x_0)$ for convenience) and considering point $(b)$, we get,
\begin{align*}
\big|&W_s^{1n}-W_s^{1m}\big|^2+ \int_s^T \abs{Z_r^{1n}-Z_r^{1m}}^2dr\\
&= 2\int_s^T\left(W_r^{1n}-W_r^{1m}\right)\left(H_1^n\left(r,
X_r^{0,x_0},Z_r^{1n},Z_r^{2n}\right) - H_1^m\left(r,
X_r^{0,x_0},Z_r^{1m},Z_r^{2m}\right)\right)dr\\
&\quad -2\int_s^T(W_r^{1n}-W_r^{1m})(Z_r^{1n}-Z_r^{1m})dB_r\\
&\leq C\int_s^T
\abs{W_r^{1n}-W_r^{1m}}\left(\left(\abs{Z_r^{1n}}+\abs{Z_r^{1m}}\right)\left(1+\abs{X_r^{0,x_0}}\right)+\left(1+\abs{X_r^{0,x_0}}^{\gamma}\right)\right)dr\\
&\quad -2\int_s^T(W_r^{1n}-W_r^{1m})(Z_r^{1n}-Z_r^{1m})dB_r.
\end{align*}

\noindent But for any $x,y,z\in \R$, $|xyz|\leq \frac{1}{p}|x|^p+\frac{1}{q}|y|^q+\frac{1}{r}|z|^r$ with 
$\frac{1}{p}+\frac{1}{q}+\frac{1}{r}=1$. Then for any $\epsilon >0$ we have,
\begin{equation}\label{ineq5}
\begin{array}{ll}
\big|W_s^{1n}-W_s^{1m}\big|^2+ \int_s^T \abs{Z_r^{1n}-Z_r^{1m}}^2dr\\\\
\quad \leq \!C\!\Big\{\frac{\epsilon^2}{2}
\int_s^T\!
(\abs{Z_r^{1n}}\!+\!\abs{Z_r^{1m}})^2dr+
\frac{\epsilon^4}{4}
\int_s^T\left(1\!+\!\abs{X_r^{0,x_0}}\right)^4dr\!+\!
\frac{1}{4\epsilon^8}\int_s^T|W_r^{1n}\!-\!W_r^{1m}|^4dr+
\\\\\qquad\int_s^T|W_r^{1n}-W_r^{1m}|
\left(1+\abs{X_r^{0,x_0}}^{\gamma}\right)
dr \Big\}-2\int_s^T(W_r^{1n}-W_r^{1m})(Z_r^{1n}-Z_r^{1m})dB_r.
\end{array}
\end{equation}
Taking now $s=0$, expectation on both sides and the limit w.r.t. $n$ and $m$ we deduce that, 
\[
\limsup_{n,m\rightarrow \infty}\be[\int_0^T \abs{Z_r^{1n}-Z_r^{1m}}^2dr]\leq C\{\frac{\epsilon^2}{2}+\frac{\epsilon^4}{4}\}
\]
\no due to estimates (\ref{estimate of Z}), (\ref{estimate of X1}) and the convergence 
of (\ref{w1n hk cov}). As $\epsilon$ is arbitrary then the sequence $(Z^{1n})_{n\geq 0}$ is convergent in $\mathcal{H}^2_T(\R^m)$ to a process $Z^1$.
\ms

Next once more going back to inequality (\ref{ineq5}), taking the supremum and using BDG's inequality we deduce that,
\begin{align*}
&\be[\sup_{0\leq s\leq T}\big|W_s^{1n}-W_s^{1m}\big|^2+ \int_0^T \abs{Z_r^{1n}-Z_r^{1m}}^2dr]\\
&\!\leq C\be\!\big\{\!\frac{\epsilon^2}{2}
\int_0^T\!
(\abs{Z_r^{1n}}\!+\!\abs{Z_r^{1m}})^2dr\!+\!
\frac{\epsilon^4}{4}
\int_0^T\left(1\!+\!\abs{X_r^{0,x_0}}\right)^4dr\!+\!
\frac{1}{4\epsilon^8}\int_0^T|W_r^{1n}\!-\!W_r^{1m}|^4dr
\\&\quad +\int_0^T|W_r^{1n}-W_r^{1m}|
\left(1+\abs{X_r^{0,x_0}}^{\gamma}\right)
ds \big\}+
\frac{1}{4}
\be[\sup_{0\leq s\leq T}\big|W_s^{1n}-W_s^{1m}\big|^2]\\
&\quad+4\be[
\int_0^T \abs{Z_r^{1n}-Z_r^{1m}}^2dr].
\end{align*} which implies that,
\begin{align*}
\limsup_{n,m\rightarrow \infty}\be[\sup_{0\leq s\leq T}\big|W_s^{1n}-W_s^{1m}\big|^2]=0
\end{align*}
since $\epsilon$ is arbitrary. Thus the sequence of processes $(W^{1n})_{n\geq 1}$ converges in $\mathcal{S}^2_T(\R)$ to $W^1$ which is a continuous process. 

Finally note that we can do the same for $i=2$, i.e., we have also the convergence of 
$(Z^{2n})_{n\geq 0}$ (resp. $(W^{2n})_{n\geq 0}$) in $\mathcal{H}^2_T(\R^m)$ (resp. $\mathcal{S}^2_T(\R)$) to $Z^2$ (resp. $(W_t^2=\varpi^2(t,X^{0,x_0}_t))_{t\leq T}$). $\Box$
\medskip

\noindent \textbf{Step 5}: The limit processes $(W_s^i, Z_s^i)_{s\leq T}$, $i=1,2$, are solutions of BSDE (\ref{main BSDE starting from 0 a}). Indeed, we need to prove that (for case $i=1$),
\begin{equation}
\begin{aligned}
F^1(t, X_t^{0,x_0})= H_1\big(t,X_t^{0,x_0}, Z_t^1,
\left(u_1^*,u_2^*\right)(t,X_t^{0,x_0},Z_t^1,Z_t^2)\big)\quad
dt\otimes d\bP- a.s.
\end{aligned}
\end{equation}

\noindent For $k\geq 1$, we have,
\begin{eqnarray}\label{last step eqnarray}
&\!\!\!\!\!&\be\!\Big[\!\int_0^T\!\big|H_1^n(s, X_s^{0,x_0},
Z_s^{1n},Z_s^{2n})\!-\!H_1(s, X_s^{0,x_0}, Z_s^1, (u_1^*,u_2^*)(s,X_s^{0,x_0}, Z_s^1, Z_s^2))\big|ds\!\Big]\nonumber\\
= & & \be\Big[\int_0^T \big|H_1^n(s, X_s^{0,x_0}, Z_s^{1n},Z_s^{2n})\nonumber\\
& &\qquad \qquad -H_1(s, X_s^{0,x_0}, Z_s^{1n}, \left(u_1^*,u_2^*\right)(s,X_s^{0,x_0}, Z_s^{1n}, Z_s^{2n}))\big|\cdot1_{\{\abs{Z_s^{1n}}+\abs{Z_s^{2n}}< k\}}ds\Big]\nonumber\\
&+& \be\Big[\int_0^T \big|H_1^n(s, X_s^{0,x_0}, Z_s^{1n},Z_s^{2n})\nonumber\\
& &\qquad \qquad -H_1(s, X_s^{0,x_0}, Z_s^{1n}, \left(u_1^*,u_2^*\right)(s,X_s^{0,x_0}, Z_s^{1n}, Z_s^{2n}))\big|\cdot1_{\{\abs{Z_s^{1n}}+\abs{Z_s^{2n}}\geq k\}}ds\Big]\nonumber\\
&+& \be\Big[\int_0^T \big|H_1(s, X_s^{0,x_0}, Z_s^{1n}, \left(u_1^*,u_2^*\right)(s,X_s^{0,x_0}, Z_s^{1n}, Z_s^{2n}))\nonumber\\
& &\qquad \qquad -H_1(s, X_s^{0,x_0}, Z_s^{1}, \left(u_1^*,u_2^*\right)(s,X_s^{0,x_0}, Z_s^{1}, Z_s^{2}))\big|ds\Big].\nonumber\\
\end{eqnarray}

\noindent The first term converges to 0. Indeed, on one hand, for $n\geq 1$, point $(b)$ implies that,
\begin{align*}
|H_1^n&(s,\! X_s^{0,x_0}, \!Z_s^{1n},\!Z_s^{2n}\!)\!-\!H_1(s, \!X_s^{0,x_0},\! Z_s^{1n}, \!\left(u_1^*,\!u_2^*\right)(s,\!X_s^{0,x_0},\! Z_s^{1n},\! Z_s^{2n}))|\!\cdot\!1_{\{\abs{Z_s^{1n}}+\abs{Z_s^{2n}}< k\}}\\
&\qq\leq 
C(1+\abs{X_s^{0,x_0}})k+C(1+|X_s^{0,x_0}|^{\gamma}).
\end{align*}

\noindent On the other hand, considering point $(d)$, we obtain that,
\begin{align*}
&|H_1^n\left(\!s,\! X_s^{0,x_0}, \!Z_s^{1n},\! Z_s^{2n}\right)\!-\!H_1\left(\!s,\! X_s^{0,x_0}, \!Z_s^{1n}, \!\left(u_1^*,u_2^*\right)\!\left(\!s,\!X_s^{0,x_0},\! Z_s^{1n},\! Z_s^{2n}\!\right)\right)\!|\!\cdot\!1_{\{\abs{Z_s^{1n}}+\abs{Z_s^{2n}}< k\}}\\
&\qq\qq \leq \sup_{\{(z_s^1, z_s^2), \ \abs{z_s^1}+\abs{z_s^2}\le  k\}} |H_1^n\left(s, X_s^{0,x_0}, z_s^{1}, z_s^{2}\right)\\
&\qquad \qquad \qquad \qquad -H_1\left(s, X_s^{0,x_0}, z_s^1, \left(u_1^*,u_2^*\right)\left(s,X_s^{0,x_0}, z_s^1, z_s^2\right)\right)|\rightarrow 0 \mbox{ as }n\rw \infty.
\end{align*}

\noindent Then thanks to Lebesgue's dominated convergence theorem, the first term in (\ref{last step eqnarray}) converges to 0 in $\mathcal{H}^1_T(\R)$.
\medskip

\noindent The second term is bounded by $\frac{C}{k^{2(q-1)/q}}$ with $q\in (1,2)$.
Actually, from point $(b)$ and Markov's inequality, for $1<q<2$ 
we have, 
\begin{eqnarray*}
& &\be\Big[\int_0^T \big|H_1^n\left(s, X_s^{0,x_0}, Z_s^{1n}, Z_s^{2n}\right)\nonumber\\
& &\qquad \qquad -H_1\left(s, X_s^{0,x_0}, Z_s^{1n}, \left(u_1^*,u_2^*\right)\left(s,X_s^{0,x_0}, Z_s^{1n}, Z_s^{2n}\right)\right)\big|\cdot1_{\{\abs{Z_s^{1n}}+\abs{Z_s^{2n}}\geq k\}}ds\Big]\nonumber\\
&\leq&\!\! C\!
\Big\{\!\be\Big[\!\!\int_0^T\!\!\!\left(1\!+\!\abs{X_s^{0,x_0}}\right)^q\abs{Z_s^{1n}}^q\!+\!\left(1\!+\!\abs{X_s^{0,x_0}}^{\gamma}\right)^q\!\!ds\Big]\Big\}^{\frac{1}{q}}\!\Big\{\!\be\Big[\!\!\int_0^T\!\!\!1_{\{\abs{Z_s^{1n}}+\abs{Z_s^{2n}}\geq
k\}}ds\Big]\Big\}^{\frac{q-1}{q}}\nonumber\\
&\leq&C\Big\{\be\Big[\int_0^T1_{\{\abs{Z_s^{1n}}+\abs{Z_s^{2n}}\geq
k\}}ds\Big]\Big\}^{\frac{q-1}{q}}\nonumber\\
&\leq& \frac{C}{k^{\frac{2(q-1)}{q}}}.
\end{eqnarray*}
The second inequality is obtained by Young's inequality and estimates (\ref{estimate of Z}) and (\ref{estimate of X1}). 

The third term in (\ref{last step eqnarray}) also
converges to 0, at least for a subsequence. Actually, since the sequence $(Z^{1n})_{n\geq 1}$
converges to $Z^1$ in $\mathcal{H}_T^2(\R^m)$, there exists a subsequence $(Z^{1n_k})_{k\geq 0}$ which converges to $Z^1$, $dt\otimes d\bp$-a.e and such that $\sup_{k\geq 0}|Z^{1n_k}_t(\omega)|$ belongs to $\mathcal{H}_T^2(\R)$. Therefore, taking the continuity Assumption (A3)-(ii) of $H_1(t,x,p,
(u_1^*, u_2^*)(t,x,p,q))$ w.r.t $(p,q)$, we obtain that,
\begin{align*}
H_1&\left(s, X_s^{0,x_0}, Z_s^{1n_k},
\left(u_1^*,u_2^*\right)\left(s,X_s^{0,x_0}, Z_s^{1n_k},
Z_s^{2n_k}\right)\right)\\
&\qquad \longrightarrow_{k\rw \infty} H_1\left(s, X_s^{0,x_0}, Z_s^{1},
\left(u_1^*,u_2^*\right)\left(s,X_s^{0,x_0}, Z_s^{1},
Z_s^{2}\right)\right)\,\,  dt\otimes d\bp-a.e.
\end{align*}
\noindent and the process 
\begin{equation*}\sup_{k\geq 0}|H_1\left(s,
X_s^{0,x_0}, Z_s^{1n_k}, \left(u_1^*,u_2^*\right)\left(s,X_s^{0,x_0},
Z_s^{1n_k}, Z_s^{2n_k}\right)\right)| \in \mathcal{H}_T^q(\R).
\end{equation*}Then once more 
by the dominated convergence theorem, we obtain,
\begin{align*}
H_1&\left(s, X_s^{0,x_0}, Z_s^{1n_k},
\left(u_1^*,u_2^*\right)\left(s,X_s^{0,x_0}, Z_s^{1n_k},
Z_s^{2n_k}\right)\right) \\
&\longrightarrow_{k\rw \infty} H_1\left(s, X_s^{0,x_0}, Z_s^{1},
\left(u_1^*,u_2^*\right)\left(s,X_s^{0,x_0}, Z_s^{1},
Z_s^{2}\right)\right)\  \text{in}\  \mathcal{H}_T^q(\R),
\end{align*}
\noindent which yields to the convergence of the third term in
(\ref{last step eqnarray}).\textbf{}
\medskip

\noindent It follows that the sequence $((H_1^n(s,X_s^{0,
x_0},Z_s^{1n},Z_s^{2n})_{s\leq T})_{n\geq 1}$ converges to\\ $(H_1(s,
X_s^{0, x_0},Z_s^1,(u_1^*,u_2^*)(s, X_s^{0,x_0}, Z_s^1,Z_s^2)))_{s\leq T}$
in $L^1([0,T]\times \Omega, dt\otimes d\bP)$ and then $$F^1(s,
X_s^{0,x_0})= H_1(s, X_s^{0, x_0},Z_s^1,(u_1^*,u_2^*)(s, X_s^{0,x_0},
Z_s^1,Z_s^2)), \,\,dt\otimes d\bP-a.e.$$ In the same way we have,
$$F^2(s,
X_s^{0,x_0})= H_2(s, X_s^{0, x_0},Z_s^2,(u_1^*,u_2^*)(s, X_s^{0,x_0},
Z_s^1,Z_s^2)), \,\,dt\otimes d\bP-a.s.$$
Thus the processes $(W^i, Z^i)$, $i=1,2$, is solution of the backward
equation (\ref{main BSDE starting from 0 a}). Finally taking into account of estimate (\ref{croipoly}) and the fact that $Z^i$, $i=1,2$, belong to $\mathcal{H}^2_T(\R^m)$ complete the proof. 
\end{proof}
\bs

As a result of Theorems \ref{bsde we mainly proof} and 
\ref{th main result section 4} we obtain the main result of this paper. 
\begin{theorem}\label{theorem_existence_nep} Assume that (A1), (A2) and (A3) are in force. Then the nonzero-sum differential game defined by (\ref{new equation of x}) and (\ref{cost function}) has a Nash equilibrium point. $\Box$
\end{theorem}
\no \textbf{Example}: The linear quadratic case. 
\ms

\noindent Let us take $m=1$ and for $t\in [0,T]$, $X^{0,x_0}_t=x_0+B_t$. Let $f(t,x,u,v)=ax+bu+cv$ and for $i=1,2$, $h_i(t,x,u,v)=\theta_ix^{p_i}+\gamma_iu^2+\rho_iv^2$, $U=[-1,1]$ and $V=[0,1]$ where $a,b,c$, $p_i$, $\rho_i$ $\gamma_i$ are real constants such that 
$p_i\geq 0$, $\gamma_1>0$ and $\rho_2>0$. Finally let $g^i$, $i=1,2$, be two Borel measurable functions with polynomial growth. 
The Hamiltonian functions $H_i$ of the nonzero-sum differential game associated with $X^{0,x_0}$, $f$, $h_i$, $g_i$, $i=1,2$, and $U$, $V$ are:
$$
H_i(t,x,z_i,u,v)=z_if(t,x,u,v)+h_i(t,x,u,v),\,\,i=1,2.
$$
Next for $\eta \in \R$ let $\psi$ and $\phi$ be functions from $\R$ to $\R$ defined by: 
$$\psi(\eta):= -1_{[\eta<-1]}+
\eta 1_{[-1\leq \eta\le1]}+ 1_{[\eta>1]}
\mbox { and }\phi(\eta):=1\wedge\eta^+.
$$
Thus the functions $u^*(t,x,z_1,z_2):=\psi(-\frac{bz_1}{2\gamma_1})$ and $v^*(t,x,z_1,z_2):=\phi(-\frac{cz_2}{2\rho_2})$ (which are continuous in $(z_1,z_2)$) verify the generalized Isaacs condition (A3). Therefore, according to Theorem \ref{theorem_existence_nep}, this game has a Nash equilibrium point $(u^*(t,X^{0,x_0}_t,Z_1(t),Z_2(t)),
v^*(t,X^{0,x_0}_t,Z_1(t),Z_2(t)))_{t\leq T}.$ 
\bs

\begin{remark} \label{generalisation}For sake of simplicity we have dealt with the case of two players. However the method still work if we have more than two players. We just need a minor adaptation of the generalized Isaacs condition of (A3). Actually assume there are $N$ players $P_1,...,P_N$ ($N\geq 3$) and for $i=1,...,N$, let $U_i$ be the compact set where the controls of player $P_i$ take their values. Let $(H_i)_{i=1,...,N}$ be the Hamiltonian functions of the nonzero-sum differential game associated with $f(t,x,u_1,...,u_N)$, $h_i(t,x,u_1,...,u_N)$ and $g^i$, i.e.,
$$H_i(t,x,z_i,u_1,...,u_N):=z_i\sigma^{-1}(t,x).f(t,x,u_1,...,u_N)+h_i(t,x,u_1,...,u_N).
$$
We assume that, uniformly w.r.t. $x$, $f$ is of linear growth and $h_i$, $g^i$, $i=1,...,N$, are of polynomial growth. Next assume that generalized Isaacs condition, which reads as below, is satisfied:
\ms 

\noindent (i) There exist $N$ Borel functions $u_1^*,...,u_N^*$ defined
  on $[0,T]\times \R^{(N+1)m}$, valued respectively in $U_1,..., U_N$, such that for any $i=1,...,N$, $(t, x, z_1,...,z_N)\in [0, T]\times \R^{(N+1)m}$, we have:
$$\begin{array}{l}H_i(t, x, z_i, (u_1^*,....,u_N^*)(t,x,z_1,...,z_N))\\
 \!\leq\! 
  H_i(t, x, z_i, 
(u_1^*,....,u_{i-1}^*)(t,x,z_1,...,z_N),\!u_i,\!(u_{i+1}^*,....,u_N^*)(t,x,z_1,...,z_N)), \!\forall u_i\in U_i\,;\end{array}$$
(ii) For any fixed $(t,x)$ the mapping 
$$
(z_1,...,z_N)\longmapsto H_i(t, x, z_i, (u_1^*,....,u_N^*)(t,x,z_1,...,z_N))
$$
is continuous. 
\ms

Then, if $\sigma (t,x)$ verifies the Assumption (A1), the differential game associated with the drift $f(t,x,u_1,...,u_N)$, the instantaneous payoffs $h_i(t,x,u_1,...,u_N)$ and the terminal payoffs $g^i(x)$, $i=1,...,N$, has a Nash equilibrium point. 
\end{remark}

\no \textbf{Extension}:  In this study the main points we required are:

(i) the existence of $p>1$ such that for any pair $(u,v)\in \cM$, $\be[(\zeta_T(u,v))^p]<\infty$ where 
$\zeta_T(u,v)$ is defined as in (\ref{defzeta}) ;  

(ii) the $L^q$-domination property or its adaptation ; 

(iii) the generalized Isaacs condition. 

\no As far as those points are in force, one can expect that the NZSDG has a NEP and, e.g., one can let drop H\"ormander's condition (\ref{horm}) on $\sigma$. Actually let us consider the following example where the matrix $\sigma$ is non longer bounded and does satisfy (\ref{horm}). Assume that $m=1$ and for $(t,x)\in [0,T]\times \R$, $X^{t,x}$ is solution of:
$$
dX^{t,x}_s=X^{t,x}_sdB_s, \,\,s\in [t,T]\mbox{ and }X^{t,x}_s=x>0 \mbox{ for }s\in [0,t].
$$
Note that
\begin{equation}\label{positivity}
\forall \, s\leq T, \, X^{t,x}_s>0. 
\end{equation}
Next let $$f(s,x,u,v):=x(u_s+v_s), \,s\leq T,\, U=[-1,1],\,V=[0,1]$$ and assume that $h_1$ and $h_2$ are as in the previous example. 
Therefore the generalized Isaacs condition is satisfied with $$u^*(t,x,z_1,z_2):=\psi(-\frac{z_1}{2\gamma_1})\mbox{ and } v^*(t,x,z_1,z_2):=\phi(-\frac{z_2}{2\rho_2}),$$
where $\psi$ and $\phi$ are the functions defined above. 

Obviously for any $(u,v)\in \cM$, $\zeta_T((u_s+v_s))$ belongs to $L^p(\bp)$ for any $p>1$ since $U$ and $V$ are bounded sets. Next for any $s\in [t,T]$, we have:
$$
X^{t,x}_s=x\exp\{B_s-B_t-\frac{1}{2}(s-t)\}.
$$
Thus for $s\in ]t,T]$, the law of $X^{t,x}_s$ has a density 
$p(t,x; s,y)$ given by
$$
p(t,x; s,y):=\frac{1}{\sqrt{2\pi(s-t)}}\exp\{-\frac{1}{2(s-t)}[\ln(\frac{y}{x})+\frac{1}{2}(s-t)]^2\}1_{[y>0]}, \,\,y\in \R.
$$
So let $x_0>0$, $\d>0$ ($\delta+t<T$) and $\Phi$ defined by:
$$\Phi(s,y):=\frac{p(t,x; s,y)}{p(0,x_0; s,y)}, (s,y)\in [t+\d,T]\times \R.$$
Therefore for any $\kappa>0$, $\Phi$ belongs to 
$L^q([t+\d,T]\times [\frac{1}{\kappa},\kappa], p(0,x_0; s,y)dsdy)$ which is the adaptation of property (ii) of Definition \ref{cdtdom} on $L^q$-domination, in taking into account (\ref{positivity}). Now as the following estimate holds true:
$$
\be\Big[\int_{t+\delta}^T
\{1_{[X_s^{t,x}<\kappa^{-1}]}+
1_{[X_s^{t,x}>\kappa]}\}ds
\Big]\leq \varrho(\kappa),
$$
where the function $\varrho$ is such that $\varrho(\kappa)\rw 0$ as $\kappa\rw \infty$. Then one can conclude that the NZSDG defined with those specific data $\sigma$, $U$, $V$, $f$, $g_i$, $h_i$, 
$i=1,2$, has a Nash equilibrium point. The proof can be established in the same way as we did in the previous sections.  $\Box$

\section*{Acknowledgement} We are grateful to the two anonymous referees for useful comments and suggestions.

\end{document}